\documentclass{article}
\usepackage{amsfonts}
\usepackage{amsmath}

\newtheorem{theorem}{Theorem}
\newtheorem{acknowledgement}[theorem]{Acknowledgement}

\newtheorem{corollary}[theorem]{Corollary}

\newtheorem{definition}[theorem]{Definition}
\newtheorem{example}[theorem]{Example}

\newtheorem{lemma}[theorem]{Lemma}

\newtheorem{proposition}[theorem]{Proposition}
\newtheorem{remark}[theorem]{Remark}

\newenvironment{proof}[1][Proof]{\noindent\textbf{#1.} }{\ \rule{0.5em}{0.5em}}
\input{tcilatex}
\begin{document}

\title{Discrete homotopies and the fundamental group}
\author{Conrad Plaut \\
Department of Mathematics\\
University of Tennessee\\
Knoxville, TN 37996\\
cplaut@math.utk.edu \and Jay Wilkins \\
Department of Mathematics\\
University of Connecticut\\
196 Auditorium Road, Unit 3009\\
Storrs, CT 06269-3009\\
leonard.wilkins@uconn.edu}
\maketitle

\begin{abstract}
We generalize and strengthen the theorem of Gromov that every compact
Riemannian manifold of diameter at most $D$ has a set of generators $%
g_{1},...,g_{k}$ of length at most $2D$ and relators of the form $%
g_{i}g_{m}=g_{j}$. In particular, we obtain an explicit bound for the number 
$k$ of generators in terms of the number \textquotedblleft short
loops\textquotedblright\ at every point and the number of balls required to
cover a given semi-locally simply connected geodesic space. As a consequence
we obtain a fundamental group finiteness theorem (new even for Riemannian
manifolds) that implies the fundamental group finiteness theorems of
Anderson and Shen-Wei. Our theorem requires no curvature bounds, nor lower
bounds on volume or $1$-systole. We use the method of discrete homotopies
introduced by the first author and V. N. Berestovskii. Central to the proof
is the notion of the \textquotedblleft homotopy critical
spectrum\textquotedblright\ that is closely related to the covering and
length spectra. Discrete methods also allow us to strengthen and simplify
the proofs of some results of Sormani-Wei about the covering spectrum.

\end{abstract}

\section{Introduction}

In (\cite{G1},\cite{G2}), Gromov proved the following: If $M$ is a compact
Riemannian manifold of diameter $D$ then $\pi _{1}(M)$ has a set of
generators $g_{1},...,g_{k}$ represented by loops of length at most $2D$ and
relations of the form $g_{i}g_{m}=g_{j}$. Among uses for this theorem are
fundamental group finiteness theorems: If $\mathcal{X}$ is any collection of
spaces with a global bound $N$ for the number of elements of $\pi _{1}(M)$
represented by loops of length at most $2D$ in any $X\in \mathcal{X}$ then $%
\pi _{1}(X)$ has at most $N$ generators and $N^{3}$ relators for any $X\in 
\mathcal{X}$. Therefore there are only finitely many possible fundamental
groups of spaces in $\mathcal{X}$. This strategy was employed by Michael
Anderson (\cite{An}) to show that compact $n$-manifolds with global lower
bounds on volume and Ricci curvature, and diameter $\leq D$, have finitely
many possible fundamental groups. Shen-Wei (\cite{ShW}) reached the same
conclusion, replacing the lower volume bound by a positive uniform lower
bound on the $1$-systole (the infimum of lengths of non-null closed
geodesics).

In this paper we generalize and strengthen Gromov's theorem by giving an
explicit bound for the number $k$ of generators in terms of the number of
\textquotedblleft short loops\textquotedblright\ and the number of balls
required to cover a space, at a given scale. As a consequence we are able to
prove a curvature-free finiteness theorem (Corollary \ref{fcor}) for
fundamental groups of compact geodesic spaces that generalizes and
strengthens both of the previously mentioned finiteness theorems. Our most
general theorem (Theorem \ref{ngmv}) applies to certain deck groups $\pi
_{\varepsilon }(X)$ of covering maps that measure the fundamental group at a
given scale. As an application of Theorem \ref{ngmv} we generalize a
finiteness theorem stated by Sormani-Wei (Proposition 7.8, \cite{SW2}). Note
that their proof implicitly relies on some kind of generalization of
Gromov's theorem, although details are not given in their paper.

We will first state our main result for fundamental groups, saving the more
general theorem from which it follows until we provide a little background
concerning the method of discrete homotopies.

For any path $c$ in a metric space $X$, we define $\left\vert [c]\right\vert 
$ to be the infimum of the lengths of paths in the fixed endpoint homotopy
class of $c$. For any $L>0$, let $\Gamma (X,L)$ be the supremum, over all
possible basepoints $\ast $, of the number of distinct elements $g\in \pi
_{1}(X,\ast )$ such that $\left\vert g\right\vert \leq L$. For a compact
geodesic space there may be no rectifiable curves in the homotopy class of a
path (cf. \cite{BPS}), and certainly $\Gamma (X,L)$ need not be finite (e.g.
a geodesic Hawiian Earring). If $X$ is \textit{semilocally simply connected}
then $\left\vert g\right\vert $ and $\Gamma (X,L)$ are always both finite
(Theorem \ref{slsc}, Corollary \ref{lengthc}), and the $1$-sytole of $X$ is
positive if $X$ is not simply connected (Corollary \ref{systole}). We denote
by $C(X,r,s)$ (resp. $C(X,s)$) the minimum number of open $s$-balls required
to cover a closed $r$-ball in $X$ (resp. $X$).

\begin{theorem}
\label{fthm}Suppose $X$ is a semilocally simply connected, compact geodesic
space of diameter $D$, and let $\varepsilon >0$. Then for any choice of
basepoint, $\pi _{1}(X)$ has a set of generators $g_{1},...,g_{k}$ of length
at most $2D$ and relations of the form $g_{i}g_{m}=g_{j}$ with 
\begin{equation*}
k\leq \frac{8(D+\varepsilon )}{\varepsilon }\cdot \Gamma (X,\varepsilon
)\cdot C\left( X,\frac{\varepsilon }{4}\right) ^{\frac{8(D+\varepsilon )}{%
\varepsilon }}\text{.}
\end{equation*}

In particular, if the $1$-systole of $X$ is $\sigma >0$ then we may take 
\begin{equation*}
k\leq \frac{8(D+\sigma )}{\sigma }C\left( X,D,\frac{\sigma }{4}\right) ^{%
\frac{8(D+\sigma )}{\sigma }}\text{.}
\end{equation*}
\end{theorem}

The previously mentioned theorem of Shen-Wei is an immediate corollary of
this theorem since for any $r,s>0$, $C(X,r,s)$ is uniformly controlled in
any precompact class of spaces by Gromov's Precompactness Criterion, and the
lower bound on Ricci curvature provides the precompactness (\cite{G2}). On
the other hand, Anderson showed (Remark 2.2(2), \cite{An}) that that if $M$
has Ricci curvature $\geq -(n-1)k^{2}$, diameter $\leq D$ and volume $\geq v$
then for any basepoint, the subgroup of $\pi _{1}(M)$ generated by loops of
length less than $\frac{Dv}{v_{k}(2D)}$ has order bounded by above by $\frac{%
v_{k}(2D)}{v}$ (here $v_{k}(2D)$ is the volume of the $2D$-ball in the space
form of curvature $-k$ and dimension $n=\dim M$). In other words, for the
class of spaces with these uniform bounds, with $\varepsilon :=\frac{Dv}{%
v_{k}(2D)}$ one has $\Gamma (M,\varepsilon )\leq \frac{v_{k}(2D)}{v}$.
Therefore Theorem \ref{fthm} implies Anderson's Finiteness Theorem. In fact,
we have:

\begin{corollary}
\label{fcor}Let $\mathcal{X}$ be any Gromov-Hausdorff precompact class of
semilocally simply connected compact geodesic spaces. If there are numbers $%
\varepsilon >0$ and $N$ such that for every $X\in \mathcal{X}$, $\Gamma
(X,\varepsilon )\leq N$, then there are finitely many possible fundamental
groups for spaces in $\mathcal{X}$.
\end{corollary}

We should point out a subtle but important difference between Anderson's
final step (i.e. from Remark 2.2(2) to the finiteness theorem) and our
proof. Anderson's final step depends on the fact that the universal covering
space also has Ricci curvature $\geq -(n-1)k^{2}$ and hence one may use
Bishop's volume comparison theorem in the universal covering space. One can
\textquotedblleft translate\textquotedblright\ his argument into one that
relies instead on global control of the numbers $C(X,r,s)$ in the universal
cover. However, this approach \textit{requires that one know that the
collection of all universal covers of all spaces in the class is (pointed)
Gromov-Hausdorff precompact. }The Shen-Wei and Sormani-Wei theorems also
rely on precompactness of the universal covering spaces. But without a lower
bound on the $1$-systole it is in general impossible to conclude from
precompactness of a class of spaces that the collection of their universal
covers is precompact (Example \ref{hw}), so we cannot use this strategy in
our proof.

The basic construction of \cite{BPTG}, \cite{BPUU} goes as follows in the
special case of a metric space $X$. For $\varepsilon >0$, an $\varepsilon $-%
\textit{chain} is a finite sequence $\alpha :=\{x_{0},...,x_{n}\}$ such that 
$d(x_{i},x_{i+1})<\varepsilon $ for all $i$. We define the \textit{length}
of $\alpha $ to be $L(\alpha )=\sum\limits_{i=1}^{n}d(x_{i},x_{i-1})$, and
define the size of $\alpha $ to be $\nu (\alpha ):=n$. The \textit{reversal}
of $\alpha $ is the chain $\overline{\alpha }:=\{x_{n},...,x_{0}\}$. A 
\textit{basic move }on an $\varepsilon $-chain $\alpha $ consists of either
adding or removing a single point, as long as the resulting chain is still
an $\varepsilon $-chain. An $\varepsilon $-homotopy between $\varepsilon $%
-chains $\alpha $ and $\beta $ with the same endpoints is a finite sequence
of $\varepsilon $-chains $\left\langle \alpha =\eta _{0},\eta _{1},...,\eta
_{k}=\beta \right\rangle $ such that all $\eta _{i}$ have the same endpoints
and for all $i$, $\eta _{i}$ and $\eta _{i+1}$ differ by a basic move. The
resulting equivalence classes are denoted $[\alpha ]_{\varepsilon }$; for
simplicity we will usually write $[x_{0},...,x_{n}]_{\varepsilon }$ rather
than $[\{x_{0},...,x_{n}\}]_{\varepsilon }$. If $\alpha =\{x_{0},...,x_{n}\}$
and $\beta =\{x_{n}=y_{0},...,y_{m}\}$ are $\varepsilon $-chains then the 
\textit{concatenation }$\alpha \ast \beta $ is the $\varepsilon $-chain $%
\{x_{0},...,x_{n}=y_{0},...,y_{m}\}$. It is easy to check that there is a
well-defined operation induced by concatenation: $\left[ \alpha \right]
_{\varepsilon }\ast \left[ \beta \right] _{\varepsilon }:=[\alpha \ast \beta
]_{\varepsilon }$. We define two $\varepsilon $-loops $\lambda _{1}$ and $%
\lambda _{2}$ to be \textit{freely }$\varepsilon $-homotopic if there exist $%
\varepsilon $-chains $\alpha $ and $\beta $ starting at a common point $%
x_{0} $, such that $\alpha \ast \lambda _{1}\ast \overline{\alpha }$ is $%
\varepsilon $-homotopic to $\beta \ast \lambda _{2}\ast \overline{\beta }$.

Fix a basepoint $\ast $ in $X$ (change of basepoint is algebraically
immaterial for connected metric spaces--see \cite{BPUU}). The set of all $%
\varepsilon $-homotopy classes $[\alpha ]_{\varepsilon }$ of $\varepsilon $%
-loops starting at $\ast $ forms a group $\pi _{\varepsilon }(X)$ with
operation induced by concatenation of $\varepsilon $-loops. The group $\pi
_{\varepsilon }(X)$ can be regarded as a kind of fundamental group that
measures only \textquotedblleft holes at the scale scale of $\varepsilon $%
.\textquotedblright\ An $\varepsilon $-loop $\alpha
=\{x_{0},...,x_{n}=x_{0}\}$ that is $\varepsilon $-homotopic to the trivial
loop $\{x_{0}\}$ is called $\varepsilon $\textit{-null}.

The set of all $\varepsilon $-homotopy classes $[\alpha ]_{\varepsilon }$ of 
$\varepsilon $-chains $\alpha $ in $X$ starting at $\ast $ will be denoted
by $X_{\varepsilon }$. The \textquotedblleft endpoint
mapping\textquotedblright\ will be denoted by $\phi _{\varepsilon
}:X_{\varepsilon }\rightarrow X$. That is, if $\alpha =\{\ast
=x_{0},x_{1},...,x_{n}\}$ then $\phi _{\varepsilon }([\alpha ]_{\varepsilon
})=x_{n}$. Since $\varepsilon $-homotopic $\varepsilon $-chains always have
the same endpoints, the function $\phi _{\varepsilon }$ is well-defined. We
choose $[\ast ]_{\varepsilon }$ for the basepoint of $X_{\varepsilon }$ so $%
\phi _{\varepsilon }$ is basepoint preserving. For any $\varepsilon $-chain $%
\alpha $ in $X$, let 
\begin{equation}
\left\vert \lbrack \alpha ]_{\varepsilon }\right\vert :=\inf \{L(\gamma
):\gamma \in \lbrack \alpha ]_{\varepsilon }\}\text{.}  \label{ennn}
\end{equation}

The above definition allows us to define a metric on $X_{\varepsilon }$ so
that $\pi _{\varepsilon }(X)$ acts by isometries induced by concatenation
(Definition \ref{metdef}). When $X$ is connected, $\phi _{\varepsilon }$ is
a regular covering map with deck group $\pi _{\varepsilon }(X)$, and when $X$
is geodesic the metric coincides with the usual lifted length metric
(Proposition \ref{metre}). For any $\delta \geq \varepsilon >0$ there is a
natural mapping $\phi _{\delta \varepsilon }:X_{\varepsilon }\rightarrow
X_{\delta }$ given by $\phi _{\delta \varepsilon }([\alpha ]_{\varepsilon
})=[\alpha ]_{\delta }$. This map is well defined because every $\varepsilon 
$-chain (resp. $\varepsilon $-homotopy) is a $\delta $-chain (resp. $\delta $%
-homotopy). One additional very important feature of geodesic spaces is that
any $\varepsilon $-chain has a \textquotedblleft midpoint
refinement\textquotedblright\ obtained by adding a midpoint between each
point in the chain and its successor (which is clearly an $\varepsilon $%
-homotopy), producing an $\frac{\varepsilon }{2}$-chain in the same $%
\varepsilon $-homotopy class. Refinement is often essential for arguments
involving limits, since being an $\varepsilon $-chain is not a closed
condition. For this reason, many arguments in this paper do not carry over
to general metric spaces.

The main relationship between $\pi _{\varepsilon }(X)$ and the fundamental
group $\pi _{1}(X)$ involves a function $\Lambda $ defined as follows (see
also Proposition 78, \cite{BPUU}). Given any continuous path $%
c:[0,1]\rightarrow X$, choose $0=t_{0}<\cdot \cdot \cdot <t_{n}=1$ fine
enough that every image $c([t_{i},t_{i+1}])$ is contained in the open ball $%
B(c(t_{i}),\varepsilon )$. Then $\Lambda
([c]):=[c(t_{0}),...,c(t_{n})]_{\varepsilon }$ is well-defined by Corollary %
\ref{homimp}. Note that $\Lambda $ is \textquotedblleft length
non-increasing\textquotedblright\ in the sense that $\left\vert \Lambda
([c])\right\vert \leq \left\vert \lbrack c]\right\vert $. Restricting $%
\Lambda $ to the fundamental group at any base point yields a homomorphism $%
\pi _{1}(X)\rightarrow \pi _{\varepsilon }(X)$. When $X$ is geodesic, $%
\Lambda $ is surjective since the successive points of an $\varepsilon $%
-loop $\lambda $ may be joined by geodesics to obtain a path loop whose
class goes to $[\alpha ]_{\varepsilon }$. The kernel of $\Lambda $ is
precisely described in \cite{PW2}; for the purposes of this paper we need
only know that if $X$ is a compact semilcoally simply connected geodesic
space then for small enough $\varepsilon $, $\Lambda $ is a
length-preserving isomorphism (Theorem \ref{slsc}). All of our theorems
about the fundamental group are directly derived from the next theorem via $%
\Lambda $.

\begin{theorem}
\label{ngmv}Let $X$ be a compact geodesic space of diameter $D$, and $%
\varepsilon >0$. Then

\begin{enumerate}
\item $\pi _{\varepsilon }(X)$ has a finite set of generators $G=\left\{
[\gamma _{1}]_{\varepsilon },...,[\gamma _{k}]_{\varepsilon }\right\} $ such
that $L(\gamma _{i})\leq 2(D+\varepsilon )$ for all $i$, and relators of the
form $[\gamma _{i}]_{\varepsilon }[\gamma _{j}]_{\varepsilon }=[\gamma
_{m}]_{\varepsilon }$.

\item For any $L>0$ there are at most $C\left( X,\frac{\varepsilon }{4}%
\right) ^{\frac{4L}{\varepsilon }}$ distinct elements $[\alpha
]_{\varepsilon }$ of $\pi _{\varepsilon }(X)$ such that $\left\vert [\alpha
]_{\varepsilon }\right\vert <L$, and in particular we may take 
\begin{equation*}
k\leq C\left( X,\frac{\varepsilon }{4}\right) ^{\frac{8(D+\varepsilon )}{%
\varepsilon }}
\end{equation*}
in the first part.

\item Suppose, in addition, that for any basepoint $\ast $ and $0<\delta
<\varepsilon $ there are at most $M$ distinct non-trivial elements $[\alpha
]_{\delta }\in \pi _{\delta }(X)$ such that $\left\vert [\alpha ]_{\delta
}\right\vert <\varepsilon $. Then the number of generators of $\pi _{\delta
}(X)$ with relators as in the first part may be taken to be at most 
\begin{equation*}
M\left[ \frac{8\left( D+\varepsilon \right) }{\varepsilon }\right] \left[
C\left( X,\frac{\varepsilon }{4}\right) \right] ^{\frac{8\left(
D+\varepsilon \right) }{\varepsilon }}
\end{equation*}
\end{enumerate}
\end{theorem}

The proof of the second part of the theorem is a nice illustration of the
utility of discrete methods. Fix any covering $\mathcal{B}$ of $X$ by $%
N:=C(X,\frac{\varepsilon }{4})$ $\frac{\varepsilon }{4}$-balls. Applying
Lemma \ref{lestl} and a midpoint refinement, we may represent any element of 
$\pi _{\varepsilon }(X)$ by an $\frac{\varepsilon }{2}$-loop $\alpha $ such
that $\nu (\alpha )\leq \frac{4L}{\varepsilon }+2$. We may choose one $B\in 
\mathcal{B}$ containing each point in the loop. Since the first and last
balls may always be chosen to be the same (containing the basepoint), each $%
\alpha $ corresponds to a sequential choice of at most $\frac{4L}{%
\varepsilon }$ balls in $\mathcal{B}$. But Proposition \ref{close} tells us
that if any two loops share the same sequence of balls (so corresponding
points are distance $<\frac{\varepsilon }{2}$ apart), they must be $%
\varepsilon $-homotopic. So there is at most one class $[\alpha
]_{\varepsilon }$ for each such sequence of balls, and there are at most $N^{%
\frac{4L}{\varepsilon }}$ different sequences of balls.

The proof of the first part of part of the theorem requires the construction
of a metric simplicial $2$-complex called an $(\varepsilon ,\delta )$%
-chassis for a compact geodesic space $X$, which is described in the last
section. For small enough $\delta >0$, any $(\varepsilon ,\delta )$-chassis
has edge group isomorphic to $\pi _{\varepsilon }(X)$ (although the two
spaces may not have the same homotopy type!). In this way our proof of
Theorem \ref{fthm} is quite different from Gromov's proof of his theorem.
However, it is interesting to note that he exploits the fact that the set of
lengths of minimal loops representing fundamental group elements in a
compact Riemannian manifold is discrete. Our proof depends on discreteness
of what we call the \textquotedblleft homotopy critical
spectrum\textquotedblright , which is closely related to the length spectrum
and covering spectrum of Sormani-Wei. We will now describe this spectrum and
related concepts (and questions) that are of indepenent interest.

\begin{definition}
An $\varepsilon $-loop $\lambda $ in a metric space $X$ is called $%
\varepsilon $-critical if $\lambda $ is not $\varepsilon $-null, but is $%
\delta $-null for all $\delta >\varepsilon $. When an $\varepsilon $%
-critical $\varepsilon $-loop exists, $\varepsilon $ is called a homotopy
critical value; the collection of these values is called the homotopy
critical spectrum.
\end{definition}

When $X$ is a geodesic space the functions $\phi _{\varepsilon \delta
}:X_{\delta }\rightarrow X_{\varepsilon }$ are all covering maps, which are
homeomorphisms precisely if there are no critical values $\sigma $ with $%
\varepsilon >\sigma \geq \delta $ (Lemma \ref{nocrit}). In a compact
geodesic space, the homotopy critical spectrum is discrete in $(0,\infty )$
(more about this below) and therefore indicates the exact values $%
\varepsilon >0$ where the equivalence type of the $\varepsilon $-covering
maps changes.

In (\cite{SW1},\cite{SW2},\cite{SW3},\cite{SW4}), Sormani-Wei independently
studied covering maps that encode geometric information, and defined the
\textquotedblleft covering spectrum\textquotedblright\ to be those values at
which the equivalence type of the covering maps changes. They utilized a
classical construction of Spanier (\cite{S}) for locally pathwise connected
topological spaces that provides a covering map $\pi ^{\delta }:\widetilde{X}%
^{\delta }\rightarrow X$ corresponding to the open cover of a geodesic space 
$X$ by open $\delta $-balls, which they called the $\delta $-cover of $X$.
As it turns out, despite the very different construction methods, when $%
\delta =\frac{3\varepsilon }{2}$ and $X$ is a compact geodesic space, this
covering map is isometrically equivalent to our covering map $\phi
_{\varepsilon }:X_{\varepsilon }\rightarrow X$ (\cite{PW2}). In fact, $\ker
\Lambda $ is precisely the Spanier group for the open cover of $X$ by $\frac{%
3\varepsilon }{2}$-balls. It follows that in the compact case, the covering
spectrum and homotopy critical spectrum differ precisely by a factor of $%
\frac{2}{3}$. However, we will not use any prior theorems about the covering
spectrum, but rather will directly prove stronger results about the homotopy
critical spectrum. Moreover, as should be clear from the present paper and (%
\cite{PW2}), discrete methods have many advantages, including simplicity,
amenability with the Gromov-Hausdorff metric, and applicability to
non-geodesic spaces. For example, Sormani-Wei (\cite{SW2}, Theorem 4.7) show
that the covering spectrum is contained in $\frac{1}{2}$ times the length
spectrum (set of lengths of closed geodesics) for geodesic spaces with a
universal cover. We not only show that this statement is true (replacing $%
\frac{1}{2}$ by $\frac{1}{3}$ in our notation) without assuming a universal
cover, we identify precisely the very special closed geodesics that
contribute to the homotopy critical spectrum:

\begin{definition}
An essential $\varepsilon $-circle in a geodesic space consists of the image
of an arclength parameterized (path) loop of length $3\varepsilon $ that
contains an $\varepsilon $-loop that is not $\varepsilon $-null.
\end{definition}

Being an essential circle is stronger than it may seem at first: an
essential circle is the image of a closed geodesic that is not
null-homotopic, which is also a \textit{metrically embedded} circle in the
sense that its metric as a subset of $X$ is the same as the intrinsic metric
of the circle (Theorem \ref{essential}). As Example \ref{torex} shows, even
in flat tori this is not always true for the image of a closed geodesic,
even when it is the shortest path in its homotopy class. We prove:

\begin{theorem}
\label{main1}If $X$ is a compact geodesic space then $\varepsilon >0$ is a
homotopy critical value of $X$ if and only if $X$ contains an essential $%
\varepsilon $-circle.
\end{theorem}

This theorem is connected to a problem with a long history in Riemannian
geometry: to relate the spectrum of the Laplace-Beltrami operator and the
length spectrum to one another and to topological and geometric properties
of the underlying compact manifold. For example, an important open question
is whether the \textquotedblleft weak\textquotedblright\ length spectrum
(i.e. without multiplicity) is completely determined by the Laplace
spectrum. To this mix one may add the covering/homotopy critical spectrum
(with or without multiplicity, see below), which up to multiplied constant
is a subset of the length spectrum. The analog of the main question has
already been answered: de Smit, Gornet, and Sutton recently showed that the
covering spectrum is not a spectral invariant (\cite{dGS}). We are now in a
position to propose yet another spectrum: the \textit{circle spectrum},
consisting of the lengths of metrically embedded circles, which according to
Theorem \ref{main1} and Example \ref{torex} is (up to multiplied constant)
generally strictly intermediate between the homotopy critical spectrum and
the length spectrum. For multiplicity one may consider either free
homotopies or use Definition \ref{circledef}. One may take this further,
partitioning the length spectrum according to the degree to which a closed
geodesic deviates from being metrically embedded, but we will not pursue
these directions in the present paper. Also, in \cite{PW2} we show that
essential circles can be used to create a new set of generators for the
fundamental group of a compact, semilocally simply connected space, which we
conjecture has minimal cardinality.

Essential circles give a nice geometric picture, but their discrete analogs,
which we will define now, are more useful for the type of problems we are
presently considering.

\begin{definition}
An $\varepsilon $-triad in a geodesic space $X$ is a triple $%
T:=\{x_{0},x_{1},x_{2}\}$ such that $d(x_{i},x_{j})=\varepsilon $ for all $%
i\neq j$; when $\varepsilon $ is not specified we will simply refer to a
triad. We denote by $\alpha _{T}$ the loop $\{x_{0},x_{1},x_{2},x_{0}\}$. We
say that $T$ is essential if some midpoint refinement of $\alpha _{T}$ is
not $\varepsilon $-null. Essential $\varepsilon $-triads $T_{1}$ and $T_{2}$
are defined to be equivalent if a midpoint refinement of $\alpha _{T_{1}}$
is freely $\varepsilon $-homotopic to a midpoint refinement of either $%
\alpha _{T_{1}}$ or $\overline{\alpha _{T_{1}}}$.
\end{definition}

Of course $\alpha _{T}$ is not an $\varepsilon $-chain; that is why we use a
midpoint refinement. We show that if one joins the corners of an essential $%
\varepsilon $-triad by geodesics then the resulting geodesic triangle is an
essential $\varepsilon $-circle (Proposition \ref{tequil}). Conversely,
given an essential $\varepsilon $-circle, every triad on it is an essential $%
\varepsilon $-triad (Corollary \ref{tcequiv}). We may now define essential $%
\varepsilon $-circles to be equivalent if their corresponding essential $%
\varepsilon $-triads are equivalent, and Theorem \ref{main1} allows us to
define the \textit{multiplicity }of a homotopy critical value $\varepsilon $
to be the number of non-equivalent essential $\varepsilon $-triads (or $%
\varepsilon $-circles).

We prove that \textquotedblleft close\textquotedblright\ essential triads
are equivalent:

\begin{proposition}
\label{cleq}Suppose $T=\{x_{0},x_{1},x_{2}\}$ is an essential $\varepsilon $%
-triad in a geodesic space $X$ and $T^{\prime }=\{x_{0}^{\prime
},x_{1}^{\prime },x_{2}^{\prime }\}$ is any set of three points such that $%
d(x_{i},x_{i}^{\prime })<\frac{\varepsilon }{3}$ for all $i$. If $T^{\prime
} $ is an essential triad then $T^{\prime }$ is an $\varepsilon $-triad
equivalent to $T$.
\end{proposition}

Now suppose we cover a compact geodesic space $X$ by $N$ open metric balls
of radius $r$. If $T$ is an essential $\varepsilon $-triad with $\varepsilon
\geq 3r$ then there are three distinct balls $B_{1},B_{2},B_{3}$ in the
cover, each containing one of the points of the triad. By Proposition \ref%
{cleq}, any triad having one point in each of $B_{1},B_{2},B_{3}$ is either
not essential or is an $\varepsilon $-triad equivalent to $T$. We obtain:

\begin{corollary}
\label{count}Let $X$ be a compact metric space with diameter $D$ and $a>0$.
Then there are at most 
\begin{equation*}
\binom{C\left( X,\frac{a}{3}\right) }{3}
\end{equation*}%
non-equivalent essential triads that are $\varepsilon $-triads for some $%
\varepsilon \geq a$.
\end{corollary}

Naturally one wonders how optimal this estimate is and whether it can be
improved (see also Example \ref{graph}). From Gromov's Precompactness
Theorem we immediately obtain:

\begin{theorem}
\label{ghbound}Let $\mathcal{X}$ be a Gromov-Hausdorff precompact collection
of compact geodesic spaces. For every $a>0$, there is a number $N$ such that
for any $X\in \mathcal{X}$ the number of homotopy critical values of $X$
greater than $a$, counted with multiplicity, is at most $N$.
\end{theorem}

One consequence is that the homotopy critical spectrum of any compact
geodesic space is discrete in $\left( 0,\infty \right) $, which is essential
for the proof of our main theorem. In \cite{SW1}, a version of Theorem \ref%
{ghbound} is proved assuming that all spaces in question have a universal
cover. The arguments there are indirect and without an explicit bound, since
they first show that the set of corresponding covering spaces is itself
Gromov-Hausdorff pointed precompact, then proceed by contradiction.
Obtaining even better control over the distribution of critical values for
specific classes of geodesic spaces is likely to be an interesting problem.
For example, it was shown in (\cite{SW3}) that limits of compact manifolds
with non-negative Ricci curvature have finite covering spectra. The proof
depends on deep results about the local structure of such spaces (\cite{cc1},%
\cite{cc2},\cite{cc3}). That the limiting spaces have finite covering
spectra implies that they have a universal cover in the categorial sense,
but leaves open the interesting question of whether they are semilocally
simply connected.

Gromov's Betti Numbers Theorem inspires the following question: Is there a
number $C(n)$ such that if $M$ is a Riemannian $n$-manifold with nonnegative
sectional curvature then $M$ has at most $C(n)$ homotopy critical values,
counted with multiplicity?

\section{Basic Discrete Homotopy Tools}

As is typical for metric spaces, the term \textquotedblleft
geodesic\textquotedblright\ in this paper refers to an arclength
parameterized length minimizing curve (and a geodesic space is one in which
every pair of points is joined by a geodesic). This is distinguished from
the traditional term \textquotedblleft geodesic\textquotedblright\ in
Riemannian geometry, which is only a local isometry; we will refer to such a
path in this paper as \textquotedblleft locally minimizing\textquotedblright
. The term \textquotedblleft closed geodesic\textquotedblright\ will refer
to a function from a standard circle into $X$ such that the restriction to
any sufficiently small arc is an isometry onto its image. We begin with a
few results for metric spaces in general, including the definition of a
natural metric on the space $X_{\varepsilon }$. While the lifting of a 
\textit{geodesic} metric to a covering space is a well-known construction
(see below), to our knowledge Definition \ref{metdef} gives the first method
to lift the metric of a general metric space to a covering space in such a
way that the covering map is uniformly a local isometry and the deck group
acts as isometries. In a metric space $X$ we denote by $B(x,r)$ the open
metric ball $\{y:d(x,y)<r\}$.

\begin{proposition}
\label{norm}Let $\alpha $, $\beta $ be $\varepsilon $-chains such that the
endpoint of $\alpha $ is the beginning point of $\beta $. Then

\begin{enumerate}
\item (Positive Definite) $\left\vert [\alpha ]_{\varepsilon }\right\vert
\geq 0$ and $\left\vert [\alpha ]_{\varepsilon }\right\vert =0$ if and only
if $\alpha $ is $\varepsilon $-null.

\item (Triangle Inequality) $\left\vert \left[ \alpha \ast \beta \right]
_{\varepsilon }\right\vert \leq \left\vert \lbrack \alpha ]_{\varepsilon
}\right\vert +\left\vert [\beta ]_{\varepsilon }\right\vert $
\end{enumerate}
\end{proposition}

\begin{proof}
That $\left\vert [\alpha ]_{\varepsilon }\right\vert \geq 0$ and that $%
\left\vert [\alpha ]_{\varepsilon }\right\vert =0$ when $\alpha $ is $%
\varepsilon $-null are both immediate consequences of the definition. In
general, if $\left\vert [\alpha ]_{\varepsilon }\right\vert =0$ then this
means that for every $\delta >0$ there is some $\varepsilon $-chain $\xi
=\{y_{0},...,y_{n}\}$ such that $[\alpha ]_{\varepsilon }=[\xi
]_{\varepsilon }$ and $L(\xi )<\delta $. In particular we may take $\delta
<\varepsilon $. Now for any $i<j$ we have $d(y_{i},y_{j})\leq
\sum_{k=i+1}^{j}d(y_{k},y_{k-1})\leq L(\xi )<\delta <\varepsilon $ and $%
\alpha $ is $\varepsilon $-homotopic to the $\varepsilon $-chain $%
\{y_{0},y_{n}\}$. By the same argument, $d(y_{0},y_{n})<\delta $ for all $%
\delta >0$ and therefore $d(y_{0},y_{n})=0$ and $y_{0}=y_{n}$. That is, $%
\alpha $ is $\varepsilon $-homotopic to $\{y_{0}\}$.

For the Triangle Inequality, simply note that if $\alpha ^{\prime }$ is $%
\varepsilon $-homotopic to $\alpha $ and $\beta ^{\prime }$ is $\varepsilon $%
-homotopic to $\beta $, then $\alpha ^{\prime }\ast \beta ^{\prime }$ is $%
\varepsilon $-homotopic to $\alpha \ast \beta $. Therefore 
\begin{equation}
\left\vert \left[ \alpha \ast \beta \right] _{\varepsilon }\right\vert \leq
L(\alpha ^{\prime }\ast \beta ^{\prime })=L(\alpha ^{\prime })+L(\beta
^{\prime })\text{.}  \label{fake}
\end{equation}%
Passing to the infimum $\left\vert \left[ \alpha \ast \beta \right]
_{\varepsilon }\right\vert -L(\beta ^{\prime })\leq \left\vert \lbrack
\alpha ]_{\varepsilon }\right\vert \Rightarrow \left\vert \left[ \alpha \ast
\beta \right] _{\varepsilon }\right\vert -\left\vert [\beta ]_{\varepsilon
}\right\vert \leq \left\vert \lbrack \alpha ]_{\varepsilon }\right\vert $.
Similarly $\left\vert \left[ \alpha \ast \beta \right] _{\varepsilon
}\right\vert -\left\vert [\alpha ]_{\varepsilon }\right\vert \leq \left\vert
\lbrack \beta ]_{\varepsilon }\right\vert $.
\end{proof}

\begin{definition}
\label{metdef}For $[\alpha ]_{\varepsilon },[\beta ]_{\varepsilon }\in
X_{\varepsilon }$ we define 
\begin{equation*}
d([\alpha ]_{\varepsilon },[\beta ]_{\varepsilon })=\inf \{L(\kappa ):\alpha
\ast \kappa \ast \overline{\beta }\text{ is }\varepsilon \text{-null}%
\}=\left\vert \left[ \overline{\alpha }\ast \beta \right] _{\varepsilon
}\right\vert \text{.}
\end{equation*}
\end{definition}

The second equality above follows from the fact that $\alpha \ast \left( 
\overline{\alpha }\ast \beta \right) \ast \overline{\beta }$ is $\varepsilon 
$-null. Proposition \ref{norm} implies that $d$ is a metric; we will always
use this metric on $X_{\varepsilon }$.

\begin{proposition}
\label{action}Let $X$ be a metric space and $\varepsilon >0$. Then

\begin{enumerate}
\item The function $\phi _{\varepsilon }:X_{\varepsilon }\rightarrow X$
preserves distances of length less than $\varepsilon $ and is injective when
restricted to any open $\varepsilon $-ball. In particular, $\phi
_{\varepsilon }$ is an isometry onto its image when restricted to any open $%
\frac{\varepsilon }{2}$-ball.

\item For any $\varepsilon $-loop $\lambda $ at $\ast $, the function $\tau
_{\lambda }:X_{\varepsilon }\rightarrow X_{\varepsilon }$ defined by $\tau
_{\lambda }(\left[ \alpha \right] _{\varepsilon })=\left[ \lambda \ast
\alpha \right] _{\varepsilon }$ is an isometry such that $\tau _{\lambda
}\circ \phi _{\varepsilon }=\phi _{\varepsilon }$.
\end{enumerate}
\end{proposition}

\begin{proof}
As in the proof of the positive definite property, if $d([\alpha
]_{\varepsilon },[\beta ]_{\varepsilon })<\varepsilon $ then $[\overline{%
\alpha }\ast \beta ]_{\varepsilon }$ must contain the chain $\{y_{0},y_{1}\}$
with $d(y_{0},y_{1})=d(\left[ \alpha \right] _{\varepsilon },\left[ \beta %
\right] _{\varepsilon })<\varepsilon $, where $y_{0}$ and $y_{1}$ are the
endpoints of $\alpha $ and $\beta $. That $\phi _{\varepsilon }$ is
injective on any $\varepsilon $-ball was proved in greater generality in 
\cite{BPUU}, but the argument is simple enough to repeat here. If $[\alpha
]_{\varepsilon },[\beta ]_{\varepsilon }\in B([\gamma ]_{\varepsilon
},\varepsilon )$ where $\gamma =\{\ast =x_{0},...,x_{n}\}$, then we may take 
$\alpha =\gamma \ast \{x_{n},y_{0}\}$ and $\beta =\gamma \ast
\{x_{n},y_{1}\} $. Then $\phi _{\varepsilon }([\alpha ]_{\varepsilon })=\phi
_{\varepsilon }([\beta ]_{\varepsilon })$ if and only if $y_{0}=y_{1}$,
which is true if and only if $[\alpha ]_{\varepsilon }=[\beta ]_{\varepsilon
}$.

To prove the second part, note that for any $\left[ \alpha \right]
_{\varepsilon }\in X_{\varepsilon }$, 
\begin{equation*}
\tau _{\lambda }\left( \left[ \overline{\lambda }\ast \alpha \right]
_{\varepsilon }\right) =\left[ \lambda \ast \overline{\lambda }\ast \alpha %
\right] _{\varepsilon }=\left[ \alpha \right] _{\varepsilon }\text{,}
\end{equation*}%
showing that $\tau _{\lambda }$ is onto. Next, for any $\left[ \beta \right]
_{\varepsilon }\in X_{\varepsilon }$ we have 
\begin{equation*}
d\left( \tau _{\lambda }\left( \left[ \alpha \right] _{\varepsilon }\right)
,\tau _{\lambda }\left( \left[ \beta \right] _{\varepsilon }\right) \right)
=d\left( \left[ \lambda \ast \alpha \right] _{\varepsilon },\left[ \lambda
\ast \beta \right] _{\varepsilon }\right) =\left\vert \left[ \overline{%
\lambda \ast \alpha }\ast \lambda \ast \beta \right] _{\varepsilon
}\right\vert
\end{equation*}%
\begin{equation*}
=\left\vert \left[ \overline{\alpha }\ast \overline{\lambda }\ast \lambda
\ast \beta \right] _{\varepsilon }\right\vert =\left\vert \left[ \overline{%
\alpha }\ast \beta \right] _{\varepsilon }\right\vert =d(\left[ \alpha %
\right] _{\varepsilon },[\beta ]_{\varepsilon })\text{.}
\end{equation*}%
Since $\lambda \ast \alpha $ has the same endpoint as $\alpha $, we also
have that $\tau _{\lambda }\circ \phi _{\varepsilon }=\phi _{\varepsilon }$.
\end{proof}

The second part of the proposition shows that the group $\pi _{\varepsilon
}(X)$ acts by isometries on $X_{\varepsilon }$. This action is \textit{%
discrete} in the sense of \cite{PQ}; that is, if for any $[\alpha
]_{\varepsilon }$ and $\lambda $ we have that $d(\tau _{\lambda }([\alpha
]_{\varepsilon }),[\alpha ]_{\varepsilon })<\varepsilon $ then $\tau
_{\lambda }$ is the identity--i.e. $\lambda $ is $\varepsilon $-null. Being
discrete is stronger than being free and properly discontinuous, and hence
when $\phi _{\varepsilon }:X_{\varepsilon }\rightarrow X$ is surjective, $%
\phi _{\varepsilon }$ is a regular covering map with covering group $\pi
_{\varepsilon }(X)$ (via the faithful action $[\lambda ]_{\varepsilon
}\rightarrow \tau _{\lambda }$). Surjectivity of $\phi _{\varepsilon }$ for
all $\varepsilon $ is clearly equivalent to $X$ being \textquotedblleft
chain connected\textquotedblright\ in the sense that every pair of points in 
$X$ is joined by an $\varepsilon $-chain for all $\varepsilon $. Chain
connected is equivalent to what is sometimes called \textquotedblleft
uniformly connected\textquotedblright\ and is in general weaker than
connected (see \cite{BPUU} for more details).

For consistency, we observe that our metric on $X_{\varepsilon }$ is
compatible with the uniform structure defined on $X_{\varepsilon }$ in \cite%
{BPUU}. A basis for that uniform structure consists of all sets (called 
\textit{entourages}) $E_{\delta }^{\ast }$, with $0<\delta \leq \varepsilon $%
, where $E_{\delta }^{\ast }$ is defined as all ordered pairs $([\alpha
]_{\varepsilon },[\beta ]_{\varepsilon })$ such that $[\alpha \ast \overline{%
\beta }]_{\varepsilon }=[\{y,z\}]_{\varepsilon }$ for some $y,z$ with $%
d(y,z)<\delta $. That is, $E_{\delta }^{\ast }=\{([\alpha ]_{\varepsilon
},[\beta ]_{\varepsilon }):d([\alpha ]_{\varepsilon },[\beta ]_{\varepsilon
})<\delta \}$, which is a basis entourage for the uniform structure of the
metric defined in Definition \ref{metdef}. So the two uniform structures are
identical.

We next consider a useful basic result showing that uniformly close $%
\varepsilon $-chains are $\varepsilon $-homotopic.

\begin{definition}
Let $X$ be a metric space. Given $\alpha =\{x_{0},...,x_{n}\}$ and $\beta
=\{y_{0},...,y_{n}\}$ with $x_{i},y_{i}\in X$, define $\Delta (\alpha ,\beta
):=\underset{i}{\max }\{d(x_{i},y_{i})\}$. For any $\varepsilon >0$, if $%
\alpha $ is an $\varepsilon $-chain we define $E_{\varepsilon }(\alpha ):=%
\underset{i}{\min }\{\varepsilon -d(x_{i},x_{i+1})\}>0$. When no confusion
will result we will eliminate the $\varepsilon $ subscript.
\end{definition}

\begin{proposition}
\label{close}Let $X$ be a metric space and $\varepsilon >0$. If $\alpha
=\{x_{0},...,x_{n}\}$ is an $\varepsilon $-chain and $\beta
=\{x_{0}=y_{0},...,y_{n}=x_{n}\}$ is such that $\Delta (\alpha ,\beta )<%
\frac{E(\alpha )}{2}$ then $\beta $ is an $\varepsilon $-chain that is $%
\varepsilon $-homotopic to $\alpha $.
\end{proposition}

\begin{proof}
We will construct an $\varepsilon $-homotopy $\eta $ from $\alpha $ to $%
\beta $. By definition of $E(\alpha )$ and the triangle inequality, each
chain below is an $\varepsilon $-chain, and hence each step below is legal.
Here and in the future we use the upper bracket to indicate that we are
adding a point, and the lower bracket to indicate that we are removing a
point in each basic step. 
\begin{equation*}
\alpha =\{x_{0},x_{1},...,x_{n}\}\rightarrow \{x_{0},\overbrace{x_{1}}%
,x_{1},...,x_{n}\}\rightarrow \{x_{0},x_{1},\overbrace{y_{1}}%
,x_{1},...,x_{n}\}
\end{equation*}%
\begin{equation*}
\rightarrow \{x_{0},\underbrace{x_{1}},y_{1},x_{1},...,x_{n}\}\rightarrow
\{x_{0},y_{1},\underbrace{x_{1}},x_{2},...,x_{n}\}
\end{equation*}%
\begin{equation*}
\rightarrow \{x_{0},y_{1},\overbrace{x_{2}},x_{2},...,x_{n}\}\rightarrow
\{x_{0},y_{1},x_{2},\overbrace{y_{2}},x_{2},...,x_{n}\}
\end{equation*}%
\begin{equation*}
\rightarrow \{x_{0},y_{1},\underbrace{x_{2}},y_{2},x_{2},...,x_{n}\}%
\rightarrow \{x_{0},y_{1},y_{2},\underbrace{x_{2}},x_{3},...,x_{n}\}%
\rightarrow \cdot \cdot \cdot \rightarrow \beta
\end{equation*}
\end{proof}

In order to properly use Proposition \ref{close}, one needs chains of the
same size, and the next lemma helps with this.

\begin{lemma}
\label{lestl}Let $L,\varepsilon >0$ and $\alpha $ be an $\varepsilon $-chain
in a metric space $X$ with $L(\alpha )\leq L$. Then there is some $\alpha
^{\prime }\in \lbrack \alpha ]_{\varepsilon }$ such that $L(\alpha ^{\prime
})\leq L(\alpha )$ and $\nu (\alpha ^{\prime })=\left\lfloor \frac{2L}{%
\varepsilon }+1\right\rfloor $.
\end{lemma}

\begin{proof}
If $\alpha $ has one or two points then we may simply repeat $x_{0}$, if
necessary, (which doesn't increase length) to obtain $\alpha ^{\prime }$
with $\nu (\alpha ^{\prime })=\left\lfloor \frac{2L}{\varepsilon }%
+1\right\rfloor $. Otherwise, let $\alpha :=\{x_{0},...,x_{n}\}$ with $n\geq
2$. Suppose that for some $i$, $d(x_{i-1},x_{i})+d(x_{i},x_{i+1})<%
\varepsilon $. Then $d(x_{i-1},x_{i+1})<\varepsilon $ and the point $x_{i}$
may be removed to form a new $\varepsilon $-chain $\alpha _{1}$ that is $%
\varepsilon $-homotopic to $\alpha $ with $L(\alpha _{1})\leq L(\alpha )$.
After finitely many such steps, we have a chain $\alpha _{0}$ that is $%
\varepsilon $-homotopic to $\alpha $ and not longer, which either has two
points (then proceed as above), or $\alpha _{0}$ has the property that for
every $i$, $d(x_{i-1},x_{i})+d(x_{i},x_{i+1})\geq \varepsilon $. By pairing
off terms we see that $L(\alpha _{0})\geq \left\lfloor \frac{\nu (\alpha
_{0})}{2}\right\rfloor \varepsilon $ and hence 
\begin{equation*}
\nu (\alpha _{0})\leq \left\lfloor \frac{2L(\alpha _{0})}{\varepsilon }%
+1\right\rfloor \leq \left\lfloor \frac{2L}{\varepsilon }+1\right\rfloor 
\text{.}
\end{equation*}%
As before, repeat $x_{0}$ enough times to make $\alpha ^{\prime }$ with $\nu
(\alpha ^{\prime })=\left\lfloor \frac{2L}{\varepsilon }+1\right\rfloor $.
\end{proof}

The above lemma can be used like a discrete version of Ascoli's Theorem.
That is, if one has a sequence of $\varepsilon $-chains of length at most $L$
(i.e., \textquotedblleft equicontinuous\textquotedblright ) then one can
assume that all the chains have the same finite size $n$. In a compact
space, one can then choose a subsequence so that the $i^{th}$ elements in
each chain form a convergent sequence for all $0\leq i\leq n$. For example,
one may use this method in conjunction with Proposition \ref{close} to
obtain the following:

\begin{corollary}
\label{exister}If $X$ is a compact metric space, $\varepsilon >0$, and $%
\alpha $ is an $\varepsilon $-chain then there is some $\beta \in \lbrack
\alpha ]_{\varepsilon }$ such that $L(\beta )=\left\vert [\alpha
]_{\varepsilon }\right\vert $.
\end{corollary}

We next move onto the relationship between paths and chains.

\begin{definition}
Let $\alpha :=\{x_{0},...,x_{n}\}$ be an $\varepsilon $-chain in a metric
space $X$, where $\varepsilon >0$. A stringing of $\alpha $ consists of a
path $\widehat{\alpha }$ formed by concatenating paths $\gamma _{i}$ from $%
x_{i}$ to $x_{i+1}$ where each path $\gamma _{i}$ lies entirely in $%
B(x_{i},\varepsilon )$. If each $\gamma _{i}$ is a geodesic then we call $%
\widehat{\alpha }$ a chording of $\alpha $.
\end{definition}

Note that by uniform continuity, any path $c$ defined on a compact interval
may be subdivided into an $\varepsilon $-chain $\alpha $ such that $c$ is a
stringing of $\alpha $, and in any geodesic space every $\varepsilon $-chain
has a chording.

\begin{proposition}
\label{ender}If $\alpha $ is an $\varepsilon $-chain in a chain connected
metric space $X$ then the unique lift of any stringing $\widehat{\alpha }$
starting at the basepoint $[\ast ]_{\varepsilon }$ in $X_{\varepsilon }$ has 
$[\alpha ]_{\varepsilon }$ as its endpoint.
\end{proposition}

\begin{proof}
Let $\alpha _{i}:=\{x_{0},...,x_{i}\}$, with $\alpha _{n}=\alpha $. We will
prove by induction that the endpoint of the lift of a stringing $\widehat{%
\alpha _{i}}$ is $[\alpha _{i}]_{\varepsilon }$. The case $i=0$ is trivial;
suppose the statement is true for some $i<n$ and consider some stringing $%
\widehat{\alpha _{i+1}}$. Then the restriction to a segment of $\widehat{%
\alpha _{i+1}}$ is a stringing $\widehat{\alpha _{i}}$ and by the inductive
step the lift of $\widehat{\alpha _{i}}$ ends at $[\alpha _{i}]_{\varepsilon
}$. By definition of stringing, $\widehat{\alpha _{i+1}}$ is obtained from $%
\widehat{\alpha _{i}}$ by adding some path $c$ from $x_{i}$ to $x_{i+1}$
that lies entirely within $B(x_{i},\varepsilon )$. By Proposition \ref%
{action}, $\phi _{\varepsilon }$ is bijective from the set $B([\alpha
_{i}]_{\varepsilon },\varepsilon )$ onto $B(x_{i},\varepsilon )$. Therefore
the lift $\widetilde{c}$ of $c$ starting at $[\alpha _{i}]_{\varepsilon }$,
must be contained entirely in $B([\alpha _{i}]_{\varepsilon },\varepsilon )$%
. By uniqueness of lifts, the endpoint of the lift of $\widehat{\alpha _{i+1}%
}$ must be the endpoint $[\beta ]_{\varepsilon }$ of $\widetilde{c}$. Note
that $\phi _{\varepsilon }([\beta ]_{\varepsilon })=x_{i+1}$; i.e. the
endpoint of $\beta $ is $x_{i+1}$. Next, $[\beta ]_{\varepsilon }\in
B([\alpha _{i}]_{\varepsilon },\varepsilon )$ means that there is some $%
\varepsilon $-chain $\sigma =\{y_{0},...,y_{m}\}$ such that $[\alpha
_{i}]_{\varepsilon }=[y_{0},...y_{m},x_{i}]_{\varepsilon }$ and $[\beta
]_{\varepsilon }=[y_{0},...,y_{m},x_{i+1}]_{\varepsilon }$. Since $[\alpha
_{i+1}]_{\varepsilon }$ is also clearly in $B([\alpha _{i}]_{\varepsilon
},\varepsilon )$ (just take $\sigma =\alpha _{i}$) and $\phi _{\varepsilon
}([\alpha _{i+1}]_{\varepsilon })=x_{i+1}=\phi _{\varepsilon }([\beta
]_{\varepsilon })$, the injectivity of $\phi _{\varepsilon }$ on $B([\alpha
_{i}]_{\varepsilon },\varepsilon )$ shows that $[\beta ]_{\varepsilon
}=[\alpha _{i+1}]_{\varepsilon }$.
\end{proof}

\begin{corollary}
\label{homimp}If $\alpha $ and $\beta $ are $\varepsilon $-chains in a chain
connected metric space $X$ such that there exist stringings $\widehat{\alpha 
}$ and $\widehat{\beta }$ that are path homotopic then $\alpha $ and $\beta $
are $\varepsilon $-homotopic.
\end{corollary}

\begin{proof}
Choose the basepoint to be the common starting point of $\alpha $ and $\beta 
$. Since $\widehat{\alpha }$ and $\widehat{\beta }$ are path homotopic, the
endpoints $[\alpha ]_{\varepsilon }$ and $[\beta ]_{\varepsilon }$ of their
lifts must be equal.
\end{proof}

A straightforward result in elementary homotopy theory of connected, locally
path connected spaces is that two path loops $c_{1}$ and $c_{2}$ are freely
homotopic if and only if for some paths $p_{i}$ from some particular point
to the start/endpoint of $c_{i}$, $p_{1}\ast c_{1}\ast \overline{p_{1}}$ is
fixed-endpoint homotopic to $p_{2}\ast c_{2}\ast \overline{p_{2}}$. Hence
free $\varepsilon $-homotopy, as defined in the Introduction, is the correct
discrete analog of continuous free homotopy. We have used this form because
imitating the standard continuous version is notationally tricky for chains.
The following lemma will be used later, and is the discrete analog of
\textquotedblleft rotation\textquotedblright\ of a path loop in itself.

\begin{lemma}
\label{rotate}Let $\alpha :=\{x_{0},...,x_{n}=x_{0}\}$ be an $\varepsilon $%
-loop in a metric space. Then $\alpha $ is freely $\varepsilon $-homotopic
to $\alpha ^{P}:=\{x_{P(0)},x_{P(1)},...,x_{P(n-1)},x_{P(0)}\}$, where $P$
is any cyclic permutation of $\{0,1,...,n-1\}$.
\end{lemma}

\begin{proof}
It suffices to consider the cycle $P$ that adds one to each index, $\func{mod%
}(n)$. Let $\beta =\{\ast =y_{0},...,y_{m}=x_{0}\}$ be an $\varepsilon $%
-chain. Here, and in the future, we will denote $\varepsilon $-homotopies in
the following form, where bracket on top denotes insertion and bracket on
the bottom denotes deletion: 
\begin{equation*}
\beta \ast \alpha \ast \overline{\beta }%
=\{y_{0},...,y_{m},x_{1},...,x_{0},y_{m-1},...,y_{0}\}
\end{equation*}%
\begin{equation*}
\rightarrow \{y_{0},...,y_{m},x_{1},...,x_{0},\overbrace{x_{0}}%
,y_{m-1},...,y_{0}\}
\end{equation*}%
\begin{equation*}
\rightarrow \{y_{0},...,y_{m},x_{1},...,x_{0},\overbrace{x_{1}}%
,x_{0},y_{m-1},...,y_{0}\}
\end{equation*}%
\begin{equation*}
\rightarrow \{y_{0},...,y_{m},x_{1},...,x_{0},\overbrace{x_{1}}%
,x_{1},x_{0},y_{m-1},...,y_{0}\}
\end{equation*}%
That is, $\beta \ast \alpha \ast \overline{\beta }$ is $\varepsilon $%
-homotopic to $\eta \ast \alpha ^{P}\ast \eta $, where $\eta
=\{y_{0},...,y_{m}=x_{0},x_{1}\}$.
\end{proof}

The rest of this section is devoted to basic results that are true for
geodesic spaces. The situation for metric spaces in general is much more
complicated--for example, the homotopy critical spectrum of a compact metric
space may not be discrete (\cite{REU}).

The following statement is easy to check: Let $f:X\rightarrow Y$ be a
bijection between geodesic spaces $X$ and $Y$. Then the following are
equivalent: (1) $f$ is an isometry. (2) $f$ is a local isometry (i.e. for
each $x\in X$ the restriction of $f$ to some $B(x,\varepsilon )$ is an
isometry onto $B(f(x),\varepsilon )$). (3) $f$ is a length-preserving
homeomorphism (i.e. if $c$ is a rectifiable path in $X$ then $f\circ c$ is
rectifiable and $L(f\circ c)=L(c)$).

Reall that if $f:X\rightarrow Y$ is a covering map, $Y$ is a geodesic space,
and $X$ is a connected topological space, then the \textit{lifted length
metric} on $X$ is defined by $d(x,y)=\inf \{L(f\circ c)\}$, where $c$ is a
path joining $x$ and $y$. When $Y$ is proper (i.e. its closed metric balls
are compact) then $X$, being uniformly locally isometric to $Y$, is locally
compact and complete. Hence by a classical result of Cohn-Vossen, $X$ 
\textit{with the lifted length metric} is also a proper geodesic space. In
this case it also follows from what was stated previously that the lifted
metric is the unique geodesic metric on $X$ such that $f$ is a local
isometry. In particular, if $g:Z\rightarrow Y$ is a covering map, where $Z$
is geodesic, and $h:X\rightarrow Z$ is a covering equivalence then $h$ is an
isometry.

\begin{proposition}
\label{metre}If $X$ is a geodesic space then the metric on $X_{\varepsilon }$
given in Definition \ref{metdef} is the lifted length metric. In particular
if $X$ is proper then $X_{\varepsilon }$ is a proper geodesic space.
\end{proposition}

\begin{proof}
Let $[\alpha ]_{\varepsilon },[\beta ]_{\varepsilon }\in X_{\varepsilon }$.
Then 
\begin{equation*}
d([\alpha ]_{\varepsilon },[\beta ]_{\varepsilon })=\inf \{L(\kappa
):[\alpha \ast \kappa \ast \overline{\beta }]_{\varepsilon }\text{ is }%
\varepsilon \text{-null}\}
\end{equation*}%
Let $\alpha =\{x_{0},...,x_{n}\}$, $\beta =\{y_{0},...,y_{k}\}$ and $\kappa
=\{x_{n}=z_{0},...,z_{m}=y_{k}\}$. Let $\widehat{\alpha }$ and $\widehat{%
\kappa }$ be chordings of $\alpha $ and $\kappa $, and note that the length
of the chain $\kappa $ is the same as the length of the curve $\widehat{%
\kappa }$. Moreover, when $[\alpha \ast \kappa \ast \overline{\beta }%
]_{\varepsilon }$ is $\varepsilon $-null, $[\alpha \ast \kappa
]_{\varepsilon }=[\beta ]_{\varepsilon }$. We will now apply Proposition \ref%
{ender} a couple of times. First, the lift of $\widehat{\alpha }$ starting
at $[\ast ]_{\varepsilon }$ ends at $[\alpha ]_{\varepsilon }$ and the lift
of $\widehat{\alpha }\ast \widehat{\kappa }$ (which is a chording of $\alpha
\ast \kappa $) starting at $[\ast ]_{\varepsilon }$ ends at $[\alpha \ast
\kappa ]_{\varepsilon }=[\beta ]_{\varepsilon }$. By uniqueness, the lift of 
$\widehat{\alpha }\ast \widehat{\kappa }$ starting at $[\ast ]_{\varepsilon
} $ must be the concatenation of the lift of $\widehat{\alpha }$ starting at 
$[\ast ]_{\varepsilon }$ with the lift $\widetilde{\kappa }$ of $\widehat{%
\kappa }$ starting at $[\alpha ]_{\varepsilon }$. That is, $\widetilde{%
\kappa }$ is a path in $X_{\varepsilon }$ starting at $[\alpha
]_{\varepsilon }$ and ending at $[\beta ]_{\varepsilon }$, with $L(\phi
_{\varepsilon }\circ \widetilde{\kappa })=L(\widehat{\kappa })=L(\kappa )$.
This shows that the metric of Definition \ref{metdef} is a geodesic metric,
and since we already know that $\phi _{\varepsilon }$ is a local isometry,
by our previous comments on uniqueness, it must be the lifted length metric.
\end{proof}

One consequence of Proposition \ref{metre} is that each $X_{\varepsilon }$
is path connected. Then it follows from the results of \cite{BPUU} that the
maps $\phi _{\varepsilon \delta }:X_{\delta }\rightarrow X_{\varepsilon }$
are also regular covering maps (in general surjectivity is the only
question, and this requires $X_{\varepsilon }$ to be chain connected).

\begin{lemma}
\label{nocrit}If $X$ is a geodesic space then the covering map $\phi
_{\varepsilon \delta }:X_{\delta }\rightarrow X_{\varepsilon }$ is injective
if and only if there are no homotopy critical values $\sigma $ with $\delta
\leq \sigma <\varepsilon $.
\end{lemma}

\begin{proof}
If there is such a critical value $\sigma $ then there is a $\sigma $-loop $%
\lambda $ that is not $\sigma $-null but is $\varepsilon $-null. That is, $%
[\lambda ]_{\sigma }\neq \lbrack \ast ]_{\sigma }$ but $[\lambda
]_{\varepsilon }=[\ast ]_{\varepsilon }$, i.e. $\phi _{\varepsilon \sigma }$
is not injective. But since $\phi _{\sigma \delta }$ is surjective, $\phi
_{\varepsilon \delta }=\phi _{\varepsilon \sigma }\circ \phi _{\sigma \delta
}$ is not injective. Conversely, if $\phi _{\varepsilon \delta }$ is not
injective then there is some $\delta $-loop $\lambda $ that is not $\delta $%
-null but is $\varepsilon $-null. Let $\sigma :=\sup \{\tau :[\lambda
]_{\tau }\neq \lbrack \ast ]_{\tau }\}$; so $\delta \leq \sigma <\varepsilon 
$. If $\lambda $ were $\sigma $-null then any $\sigma $-null homotopy would
also be a $\tau $-homotopy for $\tau <\sigma $ sufficiently close to $\sigma 
$. So $\lambda $ is not $\sigma $-null; hence $\sigma $ is a homotopy
critical value and $\sigma <\varepsilon $.
\end{proof}

\begin{remark}
\label{gromov}Unfortunately the false statement that every free homotopy
class in a compact geodesic space has a shortest path, and this path is a
closed geodesic, is present in both editions of \cite{G1}, \cite{G2},
despite the intermediate publication of two kinds of counterexamples in (%
\cite{BPS}). In one example, a homotopy class in an infinite dimensional
\textquotedblleft weakly flat\textquotedblright\ torus contains no
rectifiable curves at all. One can also see this sort of thing in the
Hawaiian Earring with its geodesic metric. In another example in \cite{BPS}
there is a loop that is rectifiable and shortest in its homotopy class, but
is not a closed geodesic. It can easily be checked directly (and as is
equally obvious from the fact that it is not a closed geodesic), this loop
is not metrically embedded, as predicted by Theorem \ref{essential}. The
false statement in (\cite{G1}, \cite{G2}) is evidently used in the proofs of
Theorem 3.4 in \cite{SW1} and Theorem 2.7 in \cite{SW3}, although it can be
worked around using methods analogous to those found in \cite{SW4} to solve
a similar problem. Using $\varepsilon $-chains avoids all such
rectifiability issues. The next theorem clarifies the situation.
\end{remark}

\begin{theorem}
\label{slsc}If $X$ is a compact semilocally simply connected geodesic space
then the homotopy critical spectrum has a positive lower bound. If $%
\varepsilon >0$ is any such lower bound then

\begin{enumerate}
\item $\phi _{\varepsilon }:X_{\varepsilon }\rightarrow X$ is the universal
covering map of $X$.

\item The function $\Lambda $ is length preserving and hence the restriction
to $\pi _{1}(X)\rightarrow \pi _{\varepsilon }(X)$ is an isomorphism.

\item Every path has a shortest path in its fixed-endpoint homotopy class,
which is either constant or locally minimizing.

\item Every path loop has a shortest path in its free homotopy class, which
is either constant or a closed geodesic.
\end{enumerate}
\end{theorem}

\begin{proof}
Proposition 69 and Theorem 77 from \cite{BPUU} imply that for all
sufficiently small $\sigma >0$, $\phi _{\sigma }$ is the simply connected
covering map of $X$. Lemma \ref{nocrit} shows that $\phi _{\varepsilon }$ is
equivalent to $\phi _{\sigma }$ since there are no homotopy critical values
between $\varepsilon $ and $\sigma $, proving the first part. Now let $%
c:[0,a]\rightarrow X$ be a path; take $\ast =c(0)$ to be the basepoint. By
definition, $\Lambda ([c])=[\alpha ]_{\varepsilon }$, where $\alpha
:=\{c(t_{0}),...,c(t_{n})\}$ for a partition $\{0=t_{0},...,t_{n}=a\}$ that
is sufficiently fine. By Corollary \ref{exister} there is some $\beta \in
\lbrack \alpha ]_{\varepsilon }$ such that $\left\vert [\alpha
]_{\varepsilon }\right\vert =L(\beta )$. Let $c^{\prime }$ be any chording
of $\beta $; so $\left\vert [\alpha ]_{\varepsilon }\right\vert =L(\beta
)=L(c^{\prime })$. Now according to Proposition \ref{ender}, the unique
lifts of $c$ and $c^{\prime }$ at the basepoint of $X_{\varepsilon }$ have
the same endpoint, and therefore form a loop. But since $X_{\varepsilon }$
is simply connected, this means that $c$ and $c^{\prime }$ are homotopic,
and we have that $\left\vert [c]\right\vert \leq L(c^{\prime })=\left\vert
[\alpha ]_{\varepsilon }\right\vert $. Since we already have the other
inequality, the second part is finished. Moreover, we have shown that $%
\left\vert [c]\right\vert $ is actually realized by any chording of a
shortest $\varepsilon $-loop in $\Lambda ([c])$. If any segment $\sigma $ of
such a shortest loop having length at most $\frac{\varepsilon }{2}$ were not
a geodesic then the endpoints of $\sigma $ could be joined by a shorter
geodesic $\sigma ^{\prime }$. But then the loop formed by these two paths
would lie in a ball of radius $\frac{\varepsilon }{2}$, and hence would lift
as a loop. That is, we could replace $\sigma $ by $\sigma ^{\prime }$ while
staying in the same homotopy class, a contradiction. This proves the third
part, and the proof of the fourth part is similar.
\end{proof}

\begin{remark}
In the previous theorem, if $X$ is already simply connected then the proof
shows that the homotopy critical spectrum is empty (the statement of the
theorem is still correct in this case, since any real number is a lower
bound for the empty set). Conversely, if $X$ is compact and semilocally
simply connected with empty covering spectrum, then $X$ is simply connected.
The latter implication is not true without the assumption that $X$ is
semilocally simply connected (see Example \ref{empty}).
\end{remark}

\begin{definition}
\label{refinement} Let $c:[0,L]\rightarrow X$ be an arclength parameterized
path in a metric space. A subdivision $\varepsilon $-chain of $c$ is an $%
\varepsilon $-chain $\{x_{0},...,x_{n}\}$ of the form $x_{i}:=c(t_{i})$ for
some subdivision $t_{0}=0<\cdot \cdot \cdot <t_{n}=L$ such that for all $%
t_{i}$, $t_{i+1}-t_{i}<\varepsilon $ (we will refer to this condition as $%
\varepsilon $-fine). If $X$ is a geodesic space and $\alpha $ is a chain in $%
X$ then a refinement of $\alpha $ consists of a chain $\beta $ formed by
inserting between each $x_{i}$ and $x_{i+1}$ some subdivision chain of a
geodesic joining $x_{i}$ and $x_{i+1}$. If $\beta $ is an $\varepsilon $%
-chain we will call $\beta $ an $\varepsilon $-refinement of $\alpha $.
\end{definition}

Since $c$ is $1$-Lipschitz, any subdivision $\varepsilon $-chain is indeed
an $\varepsilon $-chain. Obviously a refinement of an $\varepsilon $-chain $%
\alpha $ is $\varepsilon $-homotopic to $\alpha $ (just add the points one
at a time) and hence any two refinements of $\alpha $ are $\varepsilon $%
-homotopic. A special case is the midpoint refinement defined in the
Introduction.

\begin{definition}
If $X$ is a metric space and $\varepsilon >0$, an $\varepsilon $-loop of the
form $\lambda =\alpha \ast \tau \ast \overline{\alpha }$, where $\nu (\tau
)=3$, will be called $\varepsilon $-small. Note that this notation includes
the case when $\alpha $ consists of a single point--i.e. $\lambda =\tau $.
\end{definition}

Note that any $\varepsilon $-small loop is $\varepsilon $-null, although it
may or may not be $\delta $-null for smaller $\delta $.

\begin{proposition}
\label{deltaimp}Let $X$ be a geodesic space and $0<\varepsilon <\delta $.
Suppose $\alpha ,\beta $ are $\varepsilon $-chains and $\left\langle \gamma
_{0},...,\gamma _{n}\right\rangle $ is a $\delta $-homotopy such that $%
\gamma _{0}=\alpha $ and $\gamma _{n}=\beta $. Then $[\beta ]_{\varepsilon
}=[\lambda _{1}\ast \cdot \cdot \cdot \ast \lambda _{r}\ast \alpha \ast
\lambda _{r+1}\ast \cdot \cdot \cdot \ast \lambda _{n}]_{\varepsilon }$,
where each $\lambda _{i}$ is an $\varepsilon $-refinement of a $\delta $%
-small loop.
\end{proposition}

\begin{proof}
We will prove by induction that for every $k\leq n$, an $\varepsilon $%
-refinement $\gamma _{k}^{\prime }$ of $\gamma _{k}$ is $\varepsilon $%
-homotopic to $\lambda _{1}\ast \cdot \cdot \cdot \ast \alpha \ast \cdot
\cdot \cdot \ast \lambda _{k}$, where each $\lambda _{i}$ is an $\varepsilon 
$-refinement of a $\delta $-small loop. The case $k=0$ is trivial. Suppose
the statement is true for some $0\leq k<n$. The points required to $%
\varepsilon $-refine $\gamma _{k}$ to $\gamma _{k}^{\prime }\ $will be
denoted by $m_{i}$. Suppose that $\gamma _{k+1}$ is obtained from $\gamma
_{k}$ by adding a point $x$ between $x_{i}$ and $x_{i+1}$. Let $%
\{x_{i},a_{1},...,a_{k},x\}$ be an $\varepsilon $-refinement of $\{x_{i},x\}$
and $\{x,b_{1},...,b_{m},x_{i+1}\}$ an $\varepsilon $-refinement of $%
\{x,x_{i+1}\}$, so 
\begin{equation*}
\gamma _{k+1}^{\prime
}=%
\{x_{0},m_{0},...,x_{i},a_{1},...,a_{k},x,b_{1},...,b_{m},x_{i+1},m_{r},...,x_{j}\}
\end{equation*}%
is an $\varepsilon $-refinement of $\gamma _{k+1}$. Defining $\mu
_{k+1}:=\{x_{0},m_{0},...,x_{i}\}$ and 
\begin{equation*}
\kappa
_{k+1}=\{x_{i},a_{1},...,a_{k},x,b_{1},...,b_{m},x_{i+1},m_{r},...,x_{i}\}
\end{equation*}%
we have 
\begin{equation*}
\left[ \gamma _{k+1}^{\prime }\right] _{\varepsilon }=\left[ \mu _{k+1}\ast
\kappa _{k+1}\ast \overline{\mu _{k+1}}\ast \gamma _{k}^{\prime }\right]
_{\varepsilon }
\end{equation*}%
and since the homotopy is a $\delta $-homotopy, $\lambda _{k+1}:=\mu
_{k+1}\ast \kappa _{k+1}\ast \overline{\mu _{k+1}}$ is a refinement of a $%
\delta $-small loop. The case when a point is removed from $\gamma _{k}$ is
similar, except that the $\delta $-small loop is multiplied on the right.
\end{proof}

\begin{example}
Since circles play an important role in this paper, we'll conclude this
section with the simple example of the geodesic circle $C$ of circumference $%
1$. If $\varepsilon >\frac{1}{2}$ then since all points in $C$ are of
distance at most $\frac{1}{2}$, every $\varepsilon $-loop is $\varepsilon $%
-null: just remove the points (except the endpoints) one by one. The group $%
\pi _{\varepsilon }(C)$ is trivial and $\phi _{\varepsilon }:C_{\varepsilon
}\rightarrow C$ is an isometry. On the other hand, if $\varepsilon >0$ is
fairly small, it should be intuitively clear that it is impossible to
\textquotedblleft cross the hole\textquotedblright\ with an $\varepsilon $%
-homotopy, since any basic move \textquotedblleft spans a
triangle\textquotedblright\ with side lengths equal to $\varepsilon $;
therefore $\pi _{\varepsilon }(C)$ should be the non-trivial (and in fact
Theorem \ref{slsc} tells us that it will be $\mathbb{\pi }_{1}(C)$). One can
check that in fact the homotopy critical spectrum of $C$ is $\{\frac{1}{3}\}$%
--see \cite{BCH} for a nice argument involving \textquotedblleft discrete
winding numbers\textquotedblright ; this also follows from results in the
next section.
\end{example}

\section{Essential Triads and Circles}

\begin{definition}
\label{subdivide}If $c$ is an arclength paramerterized loop, we say that $c$
is $\varepsilon $-null if every (or equivalently, some) $\varepsilon $%
-subdivision chain of $c$ is $\varepsilon $-null.
\end{definition}

\begin{lemma}
\label{lennull}Every arclength parameterized loop of length less than $%
3\varepsilon $ in a geodesic space $X$ is $\varepsilon $-null.
\end{lemma}

\begin{proof}
Let $c:[0,L]\rightarrow X$ be arclength parameterized with $c(0)=c(L)=p$ and 
$0<L<3\varepsilon $. Then there exists an $\varepsilon $-fine subdivision $%
\{0=t_{0},t_{1},t_{2},t_{3}=L\}$. Since $%
d(c(t_{1}),c(t_{3}))=d(c(t_{1}),c(t_{0}))<\varepsilon $, we may simply
remove $c(t_{2})$ and then $c(t_{1})$ to get an $\varepsilon $-null homotopy.
\end{proof}

The next corollary is proved by simply joining the points in the loop by
geodesics and concatenating them to obtain an arclength parameterized loop
of length less than $3\varepsilon $.

\begin{corollary}
\label{three}If $\lambda $ is an $\varepsilon $-loop in a geodesic space $X$
of length less than $3\varepsilon $ then $\lambda $ is $\varepsilon $-null.
\end{corollary}

\begin{remark}
If $C$ is the image of a rectifiable loop of length $L$ in a metric space $X$
then by the basic theory of curves in metric spaces, for every point $x$ on $%
C$ there are precisely two possible arclength parameterizations $%
c:[0,L]\rightarrow X$ of $C$ such that $c(0)=c(L)=x$.
\end{remark}

\begin{proposition}
\label{cireq}The image $C$ of a rectifiable path loop of length $%
L=3\varepsilon $ in a geodesic space $X$ is an essential $\varepsilon $%
-circle if and only if either arclength parameterization of it is not $%
\varepsilon $-null.
\end{proposition}

\begin{proof}
Let $c:[0,L]\rightarrow C$ be an arclength parameterization of $C$. If $C$
is not an essential $\varepsilon $-circle then by definition, every $%
\varepsilon $-chain in it is $\varepsilon $-null. But then any $\varepsilon $%
-subdivision of $c$, being an $\varepsilon $-chain, must be $\varepsilon $%
-null. Hence $c$ is by definition $\varepsilon $-null. Conversely, suppose
that $C$ is essential, and so contains an $\varepsilon $-loop $\alpha
=\{x_{0},...,x_{n}=x_{0}\}$ that is not $\varepsilon $-null, with $%
x_{i}:=c(t_{i})$. We will show that $\alpha $ is $\varepsilon $-homotopic to
a concatenation of chains that are subdivision $\varepsilon $-chains of $c$
or reversals of $c$. Then at least one of those subdivision chains must be
not $\varepsilon $-null, finishing the proof. Form a path as follows: choose
a shortest segment $\sigma _{i}$ of $c$ between $x_{i-1}$ and $x_{i}$. By
\textquotedblleft segment\textquotedblright\ we mean the restriction of $c$
to a closed interval, or a path of the form $c\mid _{\lbrack t,L]}\ast c\mid
_{\lbrack 0,s]}$ (i.e. when it is shorter to go through $x_{0}$). Let $%
\widetilde{c}:=\sigma _{1}\ast \cdot \cdot \cdot \ast \sigma _{n}$. Since
each $\sigma _{i}$ has length at most $\frac{L}{2}$, by adding points $b_{i}$
that bisect each segment $\sigma _{i}$ we see that $\alpha $ is $\varepsilon 
$-homotopic to a subdivision $\varepsilon $-chain $\widetilde{\alpha }%
:=\{x_{0},b_{1},x_{1},...,b_{n},x_{n}\}$ of $\widetilde{c}$. On the other
hand, $\widetilde{c}$ is path homotopic (in the image of $c$, in fact) to
its \textquotedblleft cancelled concatenation\textquotedblright\ $\sigma
_{1}\star \cdot \cdot \cdot \star \sigma _{n}$. Recall that the cancelled
concatenation $c_{1}\star c_{2}$ is formed by starting with the
concatenation $c_{1}\ast c_{2}$ and removing the maximal final segment of $%
c_{1}$ that is equal to an initial segment of $c_{2}$ with reversed
orientation (see \cite{BPCRT}, p. 1771, for more details). It is not hard to
check by induction that $\sigma _{1}\star \cdot \cdot \cdot \star \sigma
_{i} $ is of the form $\left( k_{1}\ast \cdot \cdot \cdot \ast k_{m}\right)
\ast d $, where the following are true: $k_{i}=c$ or $k_{i}=\overline{c}$
for all $i $ (and it is possible that $m=0$, meaning there are no $k_{i}$
factors), and for some $0\leq s<L$, $d$ is of the form $c\mid _{\lbrack
0,s]} $ or $\overline{c\mid _{\lbrack s,L]}}$. Since $\alpha $ is a loop, $%
\sigma _{1}\star \cdot \cdot \cdot \star \sigma _{n}$ has no nontrivial term 
$d$, and hence consists of concatenations of $c$ or $\overline{c}$. Since $%
\widetilde{c}$ is a stringing of $\widetilde{\alpha }$, Corollary \ref%
{homimp} implies that $\widetilde{\alpha }$, hence $\alpha $, is $%
\varepsilon $-homotopic to any subdivision $\varepsilon $-chain of $\sigma
_{1}\star \cdot \cdot \cdot \star \sigma _{n}$.
\end{proof}

A \textit{geodesic triangle} consists of three geodesics $\gamma _{1},\gamma
_{2},\gamma _{3}$ such that for some three points $v_{1},v_{2},v_{3}$, $%
\gamma _{i}$ goes from $v_{i}$ to $v_{i+1}$, with addition of vertices $%
\func{mod}3$). A geodesic triangle may be considered as a loop by taking the
arclength parameterization of the concatenation of the geodesics; as far as
being $\varepsilon $-null is concerned, the specific orientation clearly
doesn't matter. We say the triangle is $\varepsilon $-null if such a
parameterization $\varepsilon $-null.

\begin{proposition}
\label{tequil}Let $T$ be an $\varepsilon $-triad in a geodesic space. Then
any two $\varepsilon $-refinements of $\alpha _{T}$ are $\varepsilon $%
-homotopic. Moreover, the following are equivalent:

\begin{enumerate}
\item $T$ is essential.

\item No $\varepsilon $-refinement of $\alpha _{T}$ is $\varepsilon $-null.

\item Every geodesic triangle having $T$ as a vertex set is an essential $%
\varepsilon $-circle.
\end{enumerate}
\end{proposition}

\begin{proof}
Let $T:=\{x_{0},x_{1},x_{2}\}$ and $\beta
=\{x_{0},m_{0},x_{1},m_{1},x_{2},m_{2},x_{0}\}$ be a midpoint refinement of $%
\alpha _{T}$. If $m_{0}^{\prime }$ is another midpoint between $x_{0}$ and $%
x_{1}$ then the $\varepsilon $-chain $\{x_{0},m_{0},x_{1},m_{0}^{\prime
},x_{0}\}$ has length at most $2\varepsilon <3\varepsilon $ and is $%
\varepsilon $-null by Corollary \ref{three}. Therefore $\beta $ is $%
\varepsilon $-homotopic to $\{x_{0},m_{0}^{\prime
},x_{1},m_{1},x_{2},m_{2},x_{0}\}$. A similar argument shows that the other
two midpoints may be replaced, up to $\varepsilon $-homotopy. In other
words, any two midpoint refinements of $\alpha _{T}$ are $\varepsilon $%
-homotopic. But any $\varepsilon $-refinement of $\alpha _{T}$ has a common
refinement with a midpoint refinement, so by the comments after Definition %
\ref{refinement}, any two $\varepsilon $-refinements of $\alpha _{T}$ are $%
\varepsilon $-homotopic.

$1\Rightarrow 2$. If $T$ is essential then by definition some midpoint
refinement of $\alpha _{T}$ is not $\varepsilon $-null. By the very first
statement of this proposition, any other $\varepsilon $-refinement of $%
\alpha _{T}$ is not $\varepsilon $-null. $2\Rightarrow 3$. Suppose $%
C:=(\gamma _{0},\gamma _{1},\gamma _{2})$ is any geodesic triangle having $T$
as a vertex set. Then the subdivision chain of $C$ consisting of the
vertices and midpoints of the geodesics is an $\varepsilon $-refinement of $%
\alpha _{T}$ and is not $\varepsilon $-null by assumption. Since $C$ also
has length $3\varepsilon $, by definition $C$ is an essential $\varepsilon $%
-circle, and $3$ is proved. $3\Rightarrow 1$. Form a geodesic triangle,
hence an essential $\varepsilon $-circle $C$, having the points of $T$ as
vertices. Then any midpoint refinement of $T$ is an $\varepsilon $%
-subdivision of $C$, which by Proposition \ref{cireq} is not $\varepsilon $%
-null. By definition, $T$ is essential.
\end{proof}

An immediate consequence of Proposition \ref{tequil} is the following:

\begin{corollary}
\label{triadee}The following statements are equivalent for two essential $%
\varepsilon $-triads $T_{1},T_{2}$ in a geodesic space:

\begin{enumerate}
\item $T_{1}$ is equivalent to $T_{2}$.

\item Every $\varepsilon $-refinement of $\alpha _{T_{1}}$ is freely $%
\varepsilon $-homotopic to every $\varepsilon $-refinement of either $\alpha
_{T_{2}}$ or $\overline{\alpha _{T_{2}}}$.

\item Some $\varepsilon $-refinement of $\alpha _{T_{1}}$ is freely $%
\varepsilon $-homotopic to some $\varepsilon $-refinement of either $\alpha
_{T_{2}}$ or $\overline{\alpha _{T_{2}}}$.
\end{enumerate}
\end{corollary}

\begin{proof}[Proof of Proposition \protect\ref{cleq}]
Note that by Corollary \ref{triadee} we may use any $\delta $-refinement in
the arguments that follow. Suppose that $T^{\prime }$ is a $\delta $-triad;
by the triangle inequality, $\delta <\frac{5}{3}\varepsilon $. Suppose first
that $\delta \geq \frac{4}{3}\varepsilon $. By the triangle inequality, $%
L(\{x_{0},x_{0}^{\prime },x_{1}^{\prime },x_{1},x_{0}\})<\frac{10}{%
\varepsilon }\varepsilon <3\delta $, and therefore any $\delta $-refinement
of this chain is $\delta $-null by Corollary \ref{three}. Since a similar
statement applies to the loops $\{x_{1},x_{1}^{\prime },x_{2}^{\prime
},x_{2},x_{1}\}$ and $\{x_{0},x_{2},x_{2}^{\prime },x_{0}^{\prime },x_{0}\}$%
, it follows that any $\delta $-refinement of $\alpha _{T^{\prime }}$ is
freely $\delta $-homotopic to a $\delta $-refinement of $\alpha _{T}$. Since 
$T$ is an essential $\varepsilon $-triad and $\varepsilon <\delta $, any
midpoint refinement of $\alpha _{T}$, and hence any midpoint refinement of $%
\alpha _{T^{\prime }}$, is $\delta $-null. That is, $T^{\prime }$ is not
essential.

Now suppose that $\delta <\frac{4}{3}\varepsilon $. By the triangle
inequality, $L(\{x_{0},x_{0}^{\prime },x_{1}^{\prime
},x_{1},x_{0}\})<3\varepsilon $ and therefore any $\varepsilon $-refinement
of this chain is $\varepsilon $-null by Corollary \ref{three}. Since a
similar statement applies to the loops $\{x_{1},x_{1}^{\prime
},x_{2}^{\prime },x_{2},x_{1}\}$ and $\{x_{0},x_{2},x_{2}^{\prime
},x_{0}^{\prime },x_{0}\}$, it follows that any $\varepsilon $-refinement of 
$\alpha _{T^{\prime }}$ is freely $\varepsilon $-homotopic to an $%
\varepsilon $-refinement of $\alpha _{T}$. Since no $\varepsilon $%
-refinement of $\alpha _{T}$ is $\varepsilon $-null, neither is any $%
\varepsilon $-refinement of $\alpha _{T^{\prime }}$. On the other hand, if $%
\sigma >\varepsilon $, $\alpha _{T}$ is $\sigma $-null and hence $\alpha
_{T^{\prime }}$ is also $\sigma $-null. Therefore if $T^{\prime }$ is an
essential triad then $T^{\prime }$ cannot be a $\sigma $-triad for any $%
\sigma >\varepsilon $. On the other hand, if $T^{\prime }$ were an essential 
$\sigma $-triad for some $\sigma <\varepsilon $ then any midpoint refinement
of $\alpha _{T^{\prime }}$ would have to be $\varepsilon $-null, a
contradiction.
\end{proof}

\begin{theorem}
\label{essential}Let $X$ be a geodesic space, $\varepsilon >0$, $%
L=3\varepsilon $ and $c:[0,L]\rightarrow X$ be arclength parameterized. If
the image of $c$ is an essential $\varepsilon $-circle $C$ then $c$ is not
null-homotopic and $C$ is metrically embedded.
\end{theorem}

\begin{proof}
That $c$ is not null-homotopic is immediate from Corollary \ref{homimp}. For
the second part we will start by showing that the restriction of $c$ to the
interval $[\frac{L}{4},\frac{3L}{4}]$ is a geodesic, hence a metric
embedding. If not then $d(c\left( \frac{L}{4}\right) ,c\left( \frac{3L}{4}%
\right) )<\frac{L}{2}$. We will get a contradiction to Proposition \ref%
{cireq} by proving that the $\varepsilon $-loop $\alpha
=\{x_{0},x_{1},x_{2},x_{3},x_{0}\}$ for the subdivision $\{0,\frac{L}{4},%
\frac{L}{2},\frac{3L}{4},L\}$ is $\varepsilon $-null. Let $m$ be a midpoint
between $x_{1}$ and $x_{3}$. By our assumption (and since $c$ is arclength
parameterized), $\xi :=\{x_{1},x_{2},x_{3},m,x_{1}\}$ is an $\varepsilon $%
-chain and has length strictly less than $L$ and hence by Corollary \ref%
{three}, is $\varepsilon $-null. By adding points one at a time we have $%
\alpha $ is $\varepsilon $-homotopic to $\left\{
x_{0},x_{1},x_{2},x_{3},m,x_{3},x_{0}\right\} $, which is $\varepsilon $%
-homotopic to 
\begin{equation*}
\{x_{0},x_{1},x_{2},x_{3},m,x_{1},m,x_{3},x_{0}\}=\{x_{0},x_{1}\}\ast \xi
\ast \{x_{1},m,x_{3},x_{0}\}
\end{equation*}%
which is $\varepsilon $-homotopic to $\beta =\{x_{0},x_{1},m,x_{3},x_{0}\}$.
But once again, since $d(x_{1},x_{3})<\frac{L}{2}$, $\beta $ is $\varepsilon 
$-null.

Now for any $s_{0}\in \lbrack 0,L]$ we may \textquotedblleft
shift\textquotedblright\ the parameterization of $c$ to a new curve $%
c_{s_{0}}:[0,L]\rightarrow X$ that is the unique arclength monotone
reparameterization of the concatenation $c\mid _{\lbrack s_{0},L]}\ast c\mid
_{\lbrack 0,s_{0}]}$. Applying the above argument for arbitrary $s_{0}$ we
obtain the following. For every $x=c(s),y=c(t)\in C$, with $s<t$, $d(x,y)$
is the mimimum of the lengths of the two curves $c\mid _{\lbrack s,t]}$ and $%
c\mid _{\lbrack t,L]}\ast c\mid _{\lbrack 0,s]}$.

Define $r:=\frac{L}{2\pi }$, and let $K$ be the standard Euclidean circle of
radius $r$ (with the geodesic metric). Now we may define $f:C\rightarrow K$
by $f(c(t))=(r\cos \frac{t}{r},r\sin \frac{t}{r})$. Given that $c$ is
arclength parameterized, and what we proved above, it is straightforward to
check that $f$ is a well-defined isometry.
\end{proof}

\begin{corollary}
\label{tcequiv}Every $\varepsilon $-triad on an essential $\varepsilon $%
-circle is essential. Moreover, if $C_{1},C_{2}$ are essential $\varepsilon $%
-circles in a geodesic space then the following are equivalent:

\begin{enumerate}
\item $C_{1}$ and $C_{2}$ have arclength parameterizations with subdivision $%
\varepsilon $-chains that are freely $\varepsilon $-homotopic.

\item For some triads $T_{i}$ on $C_{i}$, $T_{1}$ is equivalent to $T_{2}$.

\item For any triads $T_{i}$ on $C_{i}$, $T_{1}$ is equivalent to $T_{2}$.

\item For any arclength parameterizations $c_{i}$ of $C_{i}$, any
subdivision $\varepsilon $-chain of $c_{1}$ is freely $\varepsilon $%
-homotopic to any subdivision $\varepsilon $-chain of either $c_{2}$ or $%
\overline{c_{2}}$.
\end{enumerate}
\end{corollary}

\begin{proof}
A triad $T$ on $C$ must be an $\varepsilon $-triad since by Theorem \ref%
{essential}, $C$ is metrically embedded--in fact from the same theorem it
follows that the segments of $C$ between the points of $T$ must be
geodesics. Therefore the midpoints of these geodesics give a midpoint
refinement of $\alpha _{T}$ that is also an $\varepsilon $-subdivision of a
parameterization of $C$, and hence is not $\varepsilon $-null. That is, $T$
is essential.

We next show that any two triads $T=\{x_{0},x_{1},x_{2}\}$ and $T^{\prime
}=\{x_{0}^{\prime },x_{1}^{\prime },x_{2}^{\prime }\}$ on $C$ are
equivalent. First note that $T$ is equivalent to any reordering of its
points. In fact, any reordering may be obtained by a cyclic permutation
(which is covered by Lemma \ref{rotate} applied to any midpoint refinement
of $\alpha _{T}$) and/or a swap of $x_{1}$ and $x_{2}$ (which by definition
doesn't affect equivalence since it simply reverses $\alpha _{T}$). Now
applying some reordering of $T$ we may suppose that the points are arranged
around the circle in the following order: $\{x_{0},x_{0}^{\prime
},x_{1},x_{1}^{\prime },x_{2},x_{2}^{\prime },x_{0}\}$, which is an $%
\varepsilon $-refinement of $\alpha _{T}$. By Lemma \ref{rotate}, this $%
\varepsilon $-chain is freely $\varepsilon $-homotopic to $\{x_{0}^{\prime
},x_{1},x_{1}^{\prime },x_{2},x_{2}^{\prime },x_{0},x_{0}^{\prime }\}$,
which is an $\varepsilon $-refinement of $\alpha _{T^{\prime }}$. So the
first part of the corollary is finished by Corollary \ref{triadee}.

$1\Rightarrow 2$. Choose arclength parameterizations $c_{i}$ of $C_{i}$ with
subdivision $\varepsilon $-chains $\lambda _{i}$ starting at points $z_{i}$
that are freely $\varepsilon $-homotopic. Choose one of the two triads, call
it $T_{i}$, in each $C_{i}$ starting at $z_{i}$, that is also a subdivision
chain of $c_{i}$. By the comments after Definition \ref{refinement} we see
that $\lambda _{i}$ and the midpoint refinement of $\alpha _{T_{i}}$ on $%
C_{i}$ are $\varepsilon $-homotopic. Hence midpoint refinements of $\alpha
_{T_{1}}$ and $\alpha _{T_{2}}$ are freely $\varepsilon $-homotopic, so $%
T_{1}$ is equivalent to $T_{2}$. $2\Rightarrow 3$ is an immediate
consequence of the first part of this corollary. $3\Rightarrow 4$. Consider
the triads $T_{i}=\{c_{i}(0),c_{i}(\varepsilon ),c_{i}(2\varepsilon )\}$. By
reversing one of the parameterizations, if necessary, we may suppose that $%
T_{1}$ is freely $\varepsilon $-homotopic to $T_{2}$. But then midpoint
refinements of $T_{i}$ are subdivision $\varepsilon $-chains of $c_{i}$ that
are freely $\varepsilon $-homotopic. $4\Rightarrow 1$ simply follows from
the definition.
\end{proof}

\begin{definition}
\label{circledef}An essential $\varepsilon $-circle $C_{1}$ and an essential 
$\delta $-circle $C_{2}$ are said to be equivalent if $\varepsilon =\delta $
and the four equivalent conditions in the previous corollary hold. When $%
\varepsilon $ is not determined we will just refer to essential circles.
\end{definition}

\begin{proof}[Proof of Theorem \protect\ref{main1}]
If there is an essential $\varepsilon $-circle $C$ then there is an
arclength parameterization $c:[0,3\varepsilon ]\rightarrow C$. Since $c$ is
not $\varepsilon $-null, by definition a subdivision of $[0,3\varepsilon ]$
into fourths results in an $\varepsilon $-loop $\alpha $ that is not $%
\varepsilon $-null. But for any $\delta >\varepsilon $, Lemma \ref{lennull}
(applied to $\delta $) shows that $\alpha $ must be $\delta $-null for all $%
\delta >\varepsilon $ and hence has $\varepsilon $ as its critical value.

For the converse, suppose that $\lambda $ is $\varepsilon $-critical. We
will start by showing that for all $\varepsilon <\delta <2\varepsilon $
there is a midpoint refinement of a $\delta $-small loop that is not $%
\varepsilon $-null. In fact, since $\lambda $ is $\varepsilon $-critical, it
is $\delta $-null and therefore by Proposition \ref{deltaimp} can be written
as a product of midpoint refinements of $\delta $-small loops. If all of
these loops were $\varepsilon $-null, then $\lambda $ would also be $%
\varepsilon $-null, a contradiction. Now for every $i$ we may find $\left(
\varepsilon +\frac{1}{i}\right) $-small loops $\lambda _{i}=\mu _{i}\ast
\{x_{i},y_{i},z_{i},x_{i}\}\ast \overline{\mu _{i}}$ such that midpoint
subdivisions $\theta _{i}=\{x_{i},m_{i},y_{i},n_{i},z_{i},p_{i},x_{i}\}$ are
not $\varepsilon $-null. By choosing a subsequence if necessary, we may
suppose that all six sequences converge to a limiting midpoint subdivision
chain $\mu =\{x,m,y,n,z,p,x\}$ of length at most $3\varepsilon $. But
according to Proposition \ref{close}, for large enough $i$, $\mu $ is $%
\varepsilon $-homotopic to $\mu _{i}$, which means that $\mu $ is not $%
\varepsilon $-null. This means that the chain $\{x,y,z,x\}$ must have length
equal to $3\varepsilon $. Since $d(x,y),d(y,z),d(x,z)\leq \varepsilon $ it
follows that $\{x,y,z\}$ is a triad and hence is essential. By Proposition %
\ref{tequil}, any geodesic triangle having corners $\{x,y,z\}$ is an
essential $\varepsilon $-circle.
\end{proof}

\begin{corollary}
\label{systole}Suppose $X$ is a compact geodesic space with $1$-systole $%
\sigma _{1}$. Then

\begin{enumerate}
\item $\frac{\sigma _{1}}{3}$ is a lower bound for the homotopy critical
spectrum of $X$.

\item If $X$ is semilocally simply connected and not simply connected then $%
\sigma _{1}>0$ and $\varepsilon :=\frac{\sigma _{1}}{3}$ is the smallest
homotopy critical value of $X$.
\end{enumerate}
\end{corollary}

\begin{proof}
Every parameterized essential circle is a closed geodesic that is not
null-homotopic by Theorem \ref{essential}; the first part is immediate. If $%
X $ is semilocally simply connected and not simply connected, Theorem \ref%
{slsc} implies that for some $\varepsilon >0$, $\phi _{\varepsilon
}:X_{\varepsilon }\rightarrow X$ is the simply connected covering map of $X$
and $\varepsilon $ is the smallest homotopy critical value of $X$. By
Theorem \ref{main1}, $X$ contains an essential circle, which is the image of
a closed geodesic $\gamma $ of length $3\varepsilon $. If $\gamma $ were
null-homotopic then $\gamma $ would lift as a loop, contradicting
Proposition \ref{ender} and the fact that any subdivision $\varepsilon $%
-chain of it is not $\varepsilon $-null. This implies that $\sigma _{1}\leq
3\varepsilon $. Now $X$ can be covered by open sets with the property that
every loop in the set is null-homotopic in $X$. Therefore any loop of
diameter smaller then the Lebesgue number of this cover is by definition
contained in a set in the cover, hence null-homotopic, which implies $\sigma
_{1}>0$. Now suppose that $\delta :=\frac{\sigma _{1}}{3}<\varepsilon $. If $%
\gamma $ were a non-null homotopic closed geodesic of length $\sigma _{1}$,
then $\gamma $ could not lift as a loop to to the simply connected space $%
X_{\varepsilon }$. Hence by Proposition \ref{ender}, any $\varepsilon $%
-subdivision chain $\alpha $ of $\gamma $ has the property that $[\alpha
]_{\varepsilon }\neq \lbrack \ast ]_{\varepsilon }$. This contradicts
Corollary \ref{three}.
\end{proof}

\begin{example}
\label{torex}Let $Y$ denote the flat torus obtained by identifying the sides
of a rectangle of dimensions $0<3a\leq 3b$. When $a<b$, $a$ and $b$ are
distinct homotopy critical values: For $\varepsilon >b$, $Y_{\varepsilon }=Y$%
, for $a<\varepsilon \leq b$, $Y_{\varepsilon }$ is a flat metric cylinder
over a circle of length $3a$, and for $\varepsilon \leq a$, $Y_{\varepsilon
} $ is the plane. There are infinitely many essential $a$-circles and $b$%
-circles, but all essential $a$-circles are equivalent and all essential $b$%
-circles are equivalent (Corollary \ref{homimp}). When $a=b$, $a$ is the
only homotopy critical value; both circles \textquotedblleft
unroll\textquotedblright\ simultaneously and the covers go directly from
trivial to universal. There are still two equivalence classes of essential
circles, but since the circles have the same length, $a$ is a homotopy
critical value of multiplicity $2$. Now fix $a=b=\frac{1}{3}$ (i.e. $Y$
comes from a unit square). The closed geodesic determined by a straight path
starting at the bottom left corner of the square having a slope of $\frac{1}{%
2}$ is a Riemannian isometric embedding of a circle of length $\sqrt{5}$,
which is the shortest path in its homotopy class. However, the distance
between the images of any two antipodal points is only $\frac{1}{2}$, so
this closed geodesic is not metrically embedded, hence not an essential
circle. The diagonal of the square produces an $\varepsilon $-circle $C$
with $\varepsilon =\frac{\sqrt{2}}{3}$, which is the shortest path in its
homotopy class, is metrically embedded and not null-homotopic, but is not
essential. In fact, $C$ can be homotoped to the concatenation of the two
circles of which the torus is a product. Hence any $\varepsilon $-loop $%
\lambda $ on $C$ can be $\varepsilon $-homotoped to a loop $\lambda ^{\prime
}$ in those circles. But each of these circles is not $\varepsilon $%
-essential ($\varepsilon =\frac{\sqrt{2}}{3}>\frac{1}{3}$) so $\lambda
^{\prime }$, hence $\lambda $, is $\varepsilon $-null.

Note that if one adds a thin handle to the torus it will obstruct standard
homotopies between some essential circles, but not $\varepsilon $%
-homotopies. This shows that using traditional homotopies rather than $%
\varepsilon $-homotopies in the definition of equivalence can
\textquotedblleft overcount\textquotedblright\ multiplicity. In \cite{SW2},
the multiplicity of a number $\delta $ in the covering spectrum is defined
for compact spaces with a universal cover (in the categorial sense, not
necessarily simply connected) as the minimum number of generators of a
certain type in a certain subgroup of the \textquotedblleft revised
fundamental group\textquotedblright\ (Definition 6.1). We will not recall
the definition of these groups here because they require a universal cover
and this assumption is unneccessary for our work.
\end{example}

\begin{example}
\label{empty}We will now recall the construction of a space $V$ that is
known to contain a path loop $L$ that is homotopic to arbitrarily small
loops but is not null-homotopic (see \cite{cc} or \cite{Z}), giving it a
geodesic metric in the process. The Hawiian Earring $H$ consists of all
circles of radius $\frac{1}{i}$ in the plane centered at $(0,\frac{1}{i})$, $%
i\in \mathbb{N}$, with the subspace topology. The induced geodesic metric on 
$H$ measures the distance between any two points in $H$ as the length of the
shortest path in $H$ joining them. It is easy to check that this metric is
compatible with the subspace topology. Now take the cone on $H$, which also
has a geodesic metric compatible with the topology of the cone (see, for
example, the survey article \cite{PS} for details about geodesic metrics on
glued spaces and cones). Glue two copies of this space together at the point 
$(0,0)$ in $H$. One can check that every $\varepsilon $-loop is $\varepsilon 
$-null for every $\varepsilon $, so the homotopy critical spectrum is empty
even though the space is not simply connected. This example is related to
Corollary \ref{systole} in the following way: one wonders if the requirement
that $X$ be semilocally simply connected in the second part is required. If
the path loop $L$ mentioned above had a closed geodesic in its homotopy
class then we would have a counterexample to the second part of Corollary %
\ref{systole} with the weaker hypothesis. However, such a thing is not
guaranteed--see Remark \ref{gromov}.
\end{example}

\begin{example}
\label{hw}Let $X_{n}$ be the geodesic space consisting of circles of radii $%
\frac{1}{i}$ for $1\leq i\leq n$ joined at a point. These spaces are
Gromov-Hausdorff convergent to a geodesic Hawaiian Earring, but their
universal covers consist of infinite trees with valencies tending to
infinity, and hence are not Gromov-Hausdorff (pointed) precompact. One can
\textquotedblleft thicken\textquotedblright\ these examples into a family of
Riemannian $2$-manifolds with same property. It seems like an interesting
question to characterize when precompactness of a class of geodesic spaces
(even a single space!) implies precompactness of the collection of all
covering spaces.
\end{example}

The following example makes one wonder whether Corollary \ref{count} is
optimal.

\begin{example}
\label{graph}Let $S_{n}$ denote the space consisting of two points joined by 
$n$ edges of length $\frac{3}{2}$, with the geodesic metric. Each pair of
edges determines a circle of length $3$, so there is a single critical value 
$1$ of multiplicity $\left( 
\begin{array}{c}
n \\ 
2%
\end{array}%
\right) =\frac{1}{2}(n^{2}+n)$. On the other hand, we can cover the space
using one open $\frac{1}{3}$-ball at each of the two vertices and $2$
additional $\frac{1}{3}$-balls on each edge for a total of $2(n+1)$. The
estimate from Corollary \ref{count} is $\frac{4}{3}n^{3}+2n^{2}+\frac{2}{3}n$
and at any rate each edge requires at least one ball, so one cannot do
better than a degree $3$ polynomial. Another example that can be checked in
a similar fashion is the $1$-skeleton of a regular $n$-simplex with every
edge length equal to $1$, with the geodesic metric. In this example each
boundary of a $2$-face is isometric to a standard circle of circumference $3$%
. There is a single critical value $1$ of multiplicity $\left( 
\begin{array}{c}
n+1 \\ 
3%
\end{array}%
\right) =\frac{1}{6}(n^{3}-n)$. But any cover by open $\frac{1}{3}$-balls
will again require at least one ball for each of the $\left( 
\begin{array}{c}
n+1 \\ 
2%
\end{array}%
\right) $ edges and therefore the best that Corollary \ref{count} can
provide is a polynomial of order $6$ in $n$.
\end{example}

\section{$(\protect\varepsilon ,\protect\delta )$-Chassis}

In this section, $X$ will be a compact geodesic space of diameter $D$, $%
\varepsilon >0$ is fixed, and $0<\delta <\sigma $ will be positve numbers
with $\sigma \leq \varepsilon $, on which we will place additional
requirements to reach stronger conclusions. An $(\varepsilon ,\delta )$%
-chassis is defined to be a simplicial $2$-complex that has for its vertex
set a $\delta $-dense set $V:=\{v_{0},...,v_{m}\}$ (i.e. for every $x\in X$
there is some $v_{i}$ such that $d(x,v_{i})<\delta $). We let $v_{i}$ and $%
v_{j}$ be joined by an edge if and only if $d(v_{i},v_{j})<\varepsilon $ and
let $v_{i},v_{j},v_{k}$ span a $2$-simplex if and only if all three pairs of
vertices are joined by an edge. Next, let $K$ be the $1$-skeleton of $C$ and
denote the edge joining $v_{i}$ and $v_{j}$ by $e_{ij}$, $i<j$. Define the
length of $e_{ij}$ to be $d(v_{i},v_{j})$ (distance in $X$), the length of
an edge path to be the sum of the lengths of its edges, and the simplicial
distance $d_{S}(v_{i},v_{j})$ between vertices $v_{i}\neq v_{j}$ to be the
length of a shortest edge path joining them.

Every edge path in $C$ starting at $v_{0}$ (which we take for the basepoint)
is equivalent to a chain of vertices $\{v_{0}=v_{1_{0}},...,v_{i_{k}}\}$,
which has a corresponding $\varepsilon $-chain $%
\{v_{0}=v_{1_{0}},...,v_{i_{k}}\}$ in $X$. Now the basic moves in an edge
homotopy in $C$ (replacing one side of a simplex by the concatenation of the
other two, removal of an edge followed by its reversal, or vice versa)
correspond precisely to the basic moves in an $\varepsilon $-homotopy. In
other words, the function that takes the edge-homotopy class $%
[v_{0}=v_{1_{0}},...,v_{i_{k}}=v_{0}]$ of a loop to the $\varepsilon $%
-homotopy class $[v_{0}=v_{1_{0}},...,v_{i_{k}}=v_{0}]_{\varepsilon }$ is a
well-defined homomorphism $E$ from the group of edge homotopy classes of
edge loops (i.e. the edge group) $\pi _{E}(C)$ of $C$ into $\pi
_{\varepsilon }(X)$. We denote by $D_{S}$ the diameter of $C$ with the
simplicial metric.

\begin{lemma}
\label{zero}If $\delta <\frac{\sigma }{4}$ then $C$ is connected and $E$ is
surjective. In fact, if $\beta =\{v_{a},y_{1},...,y_{n-1},v_{b}\}$ is an $%
\varepsilon $-chain joining points in $V$ in $X$, then $[\beta
]_{\varepsilon }$ contains a \textquotedblleft simplicial\textquotedblright\ 
$\sigma $-chain $\alpha $ (i.e. a chain having all points in the vertex set $%
V$) such that 
\begin{equation*}
L(\alpha )\leq L(\beta )+2\left( \frac{8L(\beta )}{\sigma }\right) \delta
\end{equation*}
\end{lemma}

\begin{proof}
Given any $v_{a},v_{b}\in V$, let $c$ be a geodesic joining them in $X$. We
may subdivide $c$ into segments with endpoints $x_{k},x_{k+1}$, $0\leq k\leq
N$, of length at most $\frac{\varepsilon }{6}$. For each $m$ we may choose a
point $v_{i_{m}}\in V$ such that $d(x_{m},v_{i_{m}})<\delta $. Since $\delta
<\frac{\varepsilon }{4}$, the triangle inequality implies that $v_{i_{m}}$
and $v_{i_{m+1}}$ are joined by an edge in $C$, and and hence $v_{a}$, $%
v_{b} $ are joined by an edge path in $C$. Surjectivity will follow from the
last statement, since we may take $v_{a}=v_{b}=v_{0}$ and then resulting $%
\alpha $ is an $\varepsilon $-loop with $[\alpha ]_{\varepsilon }$ in the
image of $E$. By refinement we may suppose $\beta $ is a $\frac{\sigma }{4}$%
-chain, and applying Lemma \ref{lestl} we may assume that $n=\left\lfloor 
\frac{8L(\beta )}{\sigma }+1\right\rfloor $. For each $i$ we may choose some 
$v_{j_{i}}$ such that $d(v_{j_{i}},x_{i})<\delta $ (letting $v_{j_{0}}=v_{a}$
and $v_{j_{n}}=v_{b}$). Since $\delta <\frac{\sigma }{4}$, Proposition \ref%
{close} now implies that $\beta $ is $\sigma $-homotopic to the $\sigma $%
-chain $\alpha :=\{v_{j_{0}},...,v_{j_{n}}\}$ and hence $[\beta
]_{\varepsilon
}=E([v_{j_{0}},...,v_{j_{n}}])=[v_{j_{0}},...,v_{j_{n}}]_{\varepsilon }$.
Moreover, the triangle inequality implies that $L(\alpha )\leq L(\beta
)+2n\delta $, completing the proof.
\end{proof}

\begin{lemma}
\label{second}If $\delta <\min \{\frac{\varepsilon }{4},\frac{\varepsilon
^{2}}{32D}\}$ then for any $v_{a},v_{b}\in V$, $d(v_{a},v_{b})\leq
d_{S}(v_{a},v_{b})\leq d(v_{a},v_{b})+\frac{\varepsilon }{2}$.
\end{lemma}

\begin{proof}
The left inequality is obvious. Subdivide a geodesic in $X$ joining $%
v_{a},v_{b}$ to produce an $\varepsilon $-chain $\beta $ of length equal to $%
d(v_{a},v_{b})$. Taking $\sigma =\varepsilon $ in Lemma \ref{zero} produces
a simplical chain $\alpha $ of length at most $L(\beta )+\frac{\varepsilon }{%
2}$ joining $v_{a}$ and $v_{b}$.
\end{proof}

\begin{lemma}
\label{third}If $\phi _{\varepsilon \sigma }$ is a bijection and $\delta
<\min \left\{ \frac{\varepsilon -\sigma }{2},\frac{\sigma }{16}\right\} $
then $E$ is injective.
\end{lemma}

\begin{proof}
Suppose $[v_{0}=v_{1_{0}},...,v_{i_{k}}=v_{0}]\in \ker E$. This means that
the $\varepsilon $-chain $\alpha :=\{v_{0}=v_{1_{0}},...,v_{i_{k}}=v_{0}\}$
is $\varepsilon $-null in $X$. The problem, of course, is that the $%
\varepsilon $-null-homotopy may not involve only simplicial $\varepsilon $%
-chains and hence does not correspond to a simplicial null-homotopy in $C$.
However, by Lemma \ref{zero}, we may assume that $\alpha $ is in fact an $%
\varepsilon $-null simplicial $\sigma $-chain. By our choice of $\sigma $, $%
\alpha $ is in fact $\sigma $-null. Let $\left\langle \alpha :=\eta
_{0},...,\eta _{m}=\{v_{0}\}\right\rangle $ be a $\sigma $-homotopy and $A$
be the set of all points $a$ such that $a$ is in some chain $\eta _{i}$. For
each $a\in A$ let $a^{\prime }\in V$ be such that $d(a,a^{\prime })<\delta <%
\frac{\varepsilon -\sigma }{2}$, provided that if $a$ is already in $V$ then 
$a^{\prime }:=a$. Finally, define $\eta _{k}^{\prime
}:=\{v_{0}=x_{k1}^{\prime },...,x_{kr_{k}}^{\prime }=v_{0}\}$ whenever $\eta
_{k}:=\{v_{0}=x_{k1},...,x_{kr_{k}}=v_{0}\}$; by definition, $\eta
_{k}^{\prime }$ is a simplicial chain and since $\alpha $ is already
simplicial $\eta _{0}^{\prime }=\eta _{0}=\alpha $. Moreover, $%
d(x_{ki}^{\prime },x_{k(i+1)}^{\prime })<\sigma +2(\frac{\varepsilon -\sigma 
}{2})=\varepsilon $. That is, $\left\langle \alpha :=\eta _{0}^{\prime
},...,\eta _{m}^{\prime }=\{v_{0}\}\right\rangle $ is an $\varepsilon $%
-homotopy via simplicial chains, and so is equivalent to a simplicial
homotopy in $C$.
\end{proof}

We will now recall the well-known method of choosing generators and
relations for $\pi _{E}(C)$, while adding a geometric twist. First, we
obtain a maximal subtree $T$ of the $1$-skeleton $K$ as follows. Choose some 
$v_{k}$ of maximal simplicial distance from $v_{0}$ and connect $v_{k}$ to $%
v_{0}$ by a shortest simplicial path $\Gamma _{1}$; $\Gamma _{1}$ is the
starting point in the construction of $T$. Since $\Gamma _{1}$ is minimal it
must be simply connected, hence a tree; if it is maximal then we are done.
Otherwise there is at least one vertex not in $\Gamma _{1}$, and we choose
one, $v_{j}$, of maximal simplicial distance from $v_{0}$. Let $\Gamma _{2}$
be a minimal simplical path from $v_{j}$ to $v_{0}$. If at some point $%
\Gamma _{2}$ meets (for the first time) any vertex $w$ already in $T$, then
we replace the segment of $\Gamma _{2}$ from $w$ to $v_{0}$ by the unique
shortest segment of $\Gamma _{1}$ from $w$ to $v_{0}$. In doing so we do not
change the length of $\Gamma _{2}$ and ensure that the union of $\Gamma _{1}$
and $\Gamma _{2}$ is still a tree. We iterate this process until all
vertices are in the tree. The resulting maximal tree $T$ has the property
that every vertex $v_{j}$ in $K$ is connected to $v_{0}$ by a unique
simplicial path contained in $T$ having length at most the simplicial
diameter $D_{S}$ of $C$.

Now $\pi _{E}(C)$ has generators and relators defined as follows (see, for
example, \cite{Arm}, Section 6.4): The generators are concatenations of the
form $[g_{ij}]=[p\ast e_{ij}\ast q]$, where $e_{ij}$ is an edge that is in $%
K $ but not in $T$ and $p$ (resp. $q$) is the unique shortest simplicial
path in $T$ from $v_{j}$ to $v_{0}$ (resp. $v_{0}$ to $v_{i})$. The
relations are of the form $[g_{ij}][g_{jk}]=[g_{ik}]$, provided $%
v_{i},v_{j},v_{k}$ span a $2$-simplex in $K$ with $i<j<k$. Note that the
simplicial length of $g_{ij}$ is at most $2D_{S}+\varepsilon $.

\begin{proof}[Proof of Theorem \protect\ref{ngmv}]
Since the homotopy critical values are discrete, we may always choose $%
\sigma <\varepsilon $ so that $\phi _{\varepsilon \sigma }$ is injective. We
may then choose $\delta $ so that all of the requirements of the above
lemmas all hold. Then the resulting generators of $\pi _{E}(C)$ correspond
under the isomorphism $E$ to classes $[\gamma _{ij}]_{\varepsilon }$ in $X$
such that the length of each $\gamma _{ij}$ is at most $2D_{S}+\varepsilon +%
\frac{\varepsilon }{2}<2(D+\varepsilon )$. This proves the first part of the
theorem, and the second part was proved in the Introduction.

For the third part, we begin by choosing an $\frac{\varepsilon }{4}$-dense
set $W=\{w_{1},...,w_{s}\}$ in $X$ and an arbitrary $\delta $-chain $\mu
_{ij}$ from $w_{i}$ to $w_{j}$ with $\mu _{ji}=\overline{\mu _{ij}}$. Given
any $\delta $-loop $\lambda =\{v_{0}=x_{0},...,x_{n}=v_{0}\}$ of length at
most $2(D+\varepsilon )$, choose a subchain\ $\mu
=\{y_{0}=v_{0},...,y_{r}=v_{0}\}$ (i.e. $y_{j}=x_{i_{j}}$ for some
increasing $i_{j}$) with the following property: If $\lambda _{j}$ denotes
the $\delta $-chain $\{y_{j}=x_{i_{j}},x_{i_{j}+1},...,x_{i_{j+1}}=y_{j+1}\}$
(i.e. the \textquotedblleft segment\textquotedblright\ of $\lambda $ from $%
y_{i}$ to $y_{i+1}$) then for any $j$, $L(\lambda _{j})<\frac{\varepsilon }{4%
}$ and $L(\lambda _{j})+L(\gamma _{j+1})\geq \frac{\varepsilon }{4}$. This
can be accomplished by iteratively removing points to form the subsequence,
in a way similar to what was done in the proof of Lemma \ref{lestl}.The same
counting argument as in that proof gives us $r\leq \frac{2L(\lambda )}{%
\varepsilon }\leq \frac{8(D+\varepsilon )}{\varepsilon }$. For each $y_{j}$,
choose some $y_{j}^{\prime }\in W$ such that $d(y_{j},y_{j}^{\prime })<\frac{%
\varepsilon }{4}$. There is now a corresponding $\delta $-chain $\lambda
^{\prime }$ that is a concatenation of paths $\mu _{i_{k}j_{k}}$, where $%
y_{k}^{\prime }=w_{i_{k}}$ and $y_{k+1}=w_{j_{k}}$. Next, let $\gamma _{j}$
be a $\delta $-chain from $y_{j}^{\prime }$ to $y_{j}$ of length at most $%
\frac{\varepsilon }{4}$. It is not hard to check that $\lambda $ is $\delta $%
-homotopic to $\beta _{r}\ast \cdot \cdot \cdot \beta _{0}\ast \lambda
^{\prime }$ , where $\beta _{0}:=\overline{\lambda _{0}}\ast \gamma _{1}\ast
\mu _{i_{0}j_{0}}$ and for $k>0$, 
\begin{equation*}
\beta _{k}:=\overline{\mu _{i_{0}j_{0}}}\ast \cdot \cdot \cdot \overline{\mu
_{i_{k}j_{k}}}\ast \overline{\gamma _{k}}\ast \overline{\lambda _{k}}\ast
\gamma _{k+1}\ast \mu _{i_{k+1}j_{k+1}}\ast \cdot \cdot \cdot \ast \mu
_{i_{0}j_{0}}\text{.}
\end{equation*}%
Let us count the ways to obtain $\lambda $. First, $\lambda ^{\prime }$
corresponds to a sequential choice of $r$ elements of $W$, so there are at
most $s^{r}$ possibilities. Next, $\lambda $ is obtained from $\lambda
^{\prime }$ by $r$ concatenations, each of which involves a choice of the
element $[\overline{\gamma _{k}}\ast \overline{\lambda _{k}}\ast \gamma
_{k+1}\ast \mu _{i_{k+1}j_{k+1}}]_{\delta }\in \pi _{\delta }(X,w)$ for some 
$w\in W$ with $L(\overline{\gamma _{k}}\ast \overline{\lambda _{k}}\ast
\gamma _{k+1}\ast \mu _{i_{k+1}j_{k+1}})<\varepsilon $. So there are at most 
$r\cdot M$ distinct choices to change from $\lambda ^{\prime }$ to $\lambda $%
.
\end{proof}

From the second part of Theorem \ref{ngmv} and Theorem \ref{slsc} we may
immediately derive the following corollary:

\begin{corollary}
\label{lengthc}Let $X$ be a compact, semilocally simply connected geodesic
space. If $\varepsilon >0$ is a lower bound for the homotopy critical
spectrum of $X$ then for any $L>0$, $\Gamma (X,L)\leq C\left( X,\frac{%
\varepsilon }{4}\right) ^{\frac{4L}{\varepsilon }}$.
\end{corollary}

\begin{proof}[Proof of Theorem \protect\ref{fthm}]
By Theorem \ref{slsc}, if $\delta <\varepsilon $ is sufficiently small, the
function $\Lambda :\pi _{1}(X)\rightarrow \pi _{\delta }(X)$ is a
length-preserving isomorphism. Then the desired generators are those
corresponding to the generators of $\pi _{\delta }(X)$ given by the third
part of Theorem \ref{ngmv}, except that, \textit{a priori} those generators
have length $2(D+\delta )$. However, since $X$ is compact and semilocally
simply connected, the proof is finished by a standard application of
Ascoli's Theorem. The statement about the $1$-systole follows from Theorem %
\ref{slsc}.
\end{proof}

\begin{example}
\label{neither}The product of a circle with smaller and smaller spheres has $%
1$-systole, but not volume, bouded below. On the other hand, Vitali
Kapovitch has pointed out (see \cite{N}, Section 0.4) that examples from 
\cite{An} can be modified to have a global lower bound on volume and Ricci
curvature with $1$-systole going to $0$. These examples show that the
finiteness theorems of Anderson and Shen-Wei are independent. At the same
time, each of these examples satisfies the conditions of Theorem \ref{fthm}.
\end{example}

\begin{acknowledgement}
Thanks to Valera Berestovskii for reading an earlier version of the paper
and providing useful comments. Zach Lindsey found an error in an earlier
draft of this paper.
\end{acknowledgement}

\end{document}